\documentclass[a4paper,11pt]{article}

\usepackage{amsmath}
\usepackage{amsfonts}
\usepackage{amsthm}
\usepackage{amssymb}
\usepackage{a4wide}
\usepackage{mathrsfs}
\usepackage{epsfig}
\usepackage{palatino}
\usepackage{tikz,fp,ifthen}
\usepackage{float}
\usepackage[utf8]{inputenc}
\usepackage{url}

\newcommand{\pdfgraphics}{\ifpdf\DeclareGraphicsExtensions{.pdf,.jpg}\else\fi}
\usepackage{graphicx}

\usepackage{color}
\definecolor{citegreen}{rgb}{0,0.6,0}
\definecolor{refred}{rgb}{0.8,0,0}
\usepackage[colorlinks, citecolor=citegreen,linkcolor=refred]{hyperref}

\usepackage{mathtools}

\numberwithin{equation}{section}

\theoremstyle{plain}

\newtheorem{definition}{Definition}[section]

\newtheorem{remark}[definition]{Remark}
\newtheorem{ass}[definition]{Assumptions}

\newtheorem{proposition}[definition]{Proposition}
\newtheorem{theorem}[definition]{Theorem}
\newtheorem{lemma}[definition]{Lemma}

\usepackage{xparse}
\let\olddefinition\remark
\RenewDocumentCommand{\remark}{o}{%
  \IfNoValueTF{#1}
    {\olddefinition}
    {\olddefinition[#1]}%
  \normalfont
}

\numberwithin{equation}{section}

\newcommand{\defl}{\mathrel{=\!\!\mathop:}}
\newcommand{\defr}{\mathrel{\mathop:\!\!=}}
\newcommand{\Div}{\mathrm{div}}
\newcommand{\Id}{\mathrm{Id}}

\def\E{\mathbb{E}}

\def\N{\mathbb{N}}
\def\R{\mathbb{R}}
\renewcommand{\setminus}{\mathbin{\backslash}}

\definecolor{mygreen}{rgb}{0,0.7,0}

\begin{document}
\pdfgraphics 

\title{Qualitative Properties for a System Coupling \\Scaled Mean Curvature Flow and Diffusion}

\author{Helmut Abels
\footnote{Fakult\"at f\"ur Mathematik, Universit\"at Regensburg, Universit\"atsstrasse 31, 
93053 Regensburg, Germany \newline
\hspace*{1.6em} helmut.abels@ur.de, felicitas.buerger@ur.de, harald.garcke@ur.de (corresponding author)}
 \and Felicitas Bürger
\footnotemark[1]
\and Harald Garcke
\footnotemark[1]
}

\maketitle

\begin{abstract}
We consider a system consisting of a geometric evolution equation for a hypersurface and a parabolic equation on this evolving hypersurface. More precisely, we discuss mean curvature flow scaled with a term that depends on a quantity defined on the surface coupled to a diffusion equation for that quantity. Several properties of solutions are analyzed. Emphasis is placed on to what extent the surface in our setting qualitatively evolves similar as for the usual mean curvature flow. To this end, we show that the surface area is strictly decreasing but give an example of a surface that exists for infinite times nevertheless. Moreover, mean convexity is conserved whereas convexity is not. Finally, we construct an embedded hypersurface that develops a self-intersection in the course of time. Additionally, a formal explanation of how our equations can be interpreted as a gradient flow is included.
\end{abstract}



\section{Introduction}

We consider a system consisting of an evolving closed hypersurface $\Gamma$ and a concentration $c: \Gamma \rightarrow \R$ defined on this evolving hypersurface that satisfy the equations
\begin{subequations}\label{eq_Intro_GLS}
\begin{align}
V &= \big(G(c) - G'(c)c\big) H, \label{eq_Intro_GLS1} \\
\partial^\square c &= \Delta_\Gamma \big( G'(c) \big) + cHV. \label{eq_Intro_GLS2} 
\end{align}
\end{subequations}
Here, $H$ and $V$ denote the mean curvature and normal velocity of the hypersurface, respectively. Furthermore, 
the differential operators $\partial^\square$ and $\Delta_\Gamma$ are the normal time derivative and the Laplace-Beltrami operator (see Section \ref{GeometricSetting} for precise definitions). The function $G: \R \rightarrow \R$ is a (Gibbs) energy density and we will often use the notation
\begin{align*}
g(c) \defr G(c)-G'(c)c
\end{align*}
which appears in the right hand side of \eqref{eq_Intro_GLS1}. A solution to the system \eqref{eq_Intro_GLS} consists of both the hypersurface $\Gamma$ and the concentration $c$. In particular, we do not prescribe the evolution of the geometry but it is part of the problem. \\
These equations \eqref{eq_Intro_GLS} can be seen as a suitable characterization of the following physical situation: We are interested in how a pair $(\Gamma,c)$ evolves to decrease the (Gibbs) energy 
\begin{align}\label{eq_Intro_Energy}
\mathrm{E}\big(\Gamma(t),c(t)\big) \defr \int_{\Gamma(t)} G\big(c(t)\big) \, \mathrm{d}\mathcal{A} 
\end{align}
most efficiently while conserving the total mass 
\begin{align*}
\int_{\Gamma(t)} c(t) \, \mathrm{d}\mathcal{A} = m
\end{align*}
of the quantity whose distribution is given by $c$. Here, $\mathrm{d}\mathcal{A}$ denotes the volume element of the surface. From this formulation \eqref{eq_Intro_Energy} it is clear why the function $G$ is called a (Gibbs) energy density. It turns out that solutions of the system of equations \eqref{eq_Intro_GLS} conserve the total mass (see Theorem \ref{Massenerhaltung}) and decrease the energy \eqref{eq_Intro_Energy} (see Theorem \ref{energy_decrease}). Even more, we show that formally, our system of equations \eqref{eq_Intro_GLS} is a gradient flow of the energy functional \eqref{eq_Intro_Energy} (see Section \ref{Sec_GradientFlow}). \\
In the following, we give an introduction to both of the equations \eqref{eq_Intro_GLS1} and \eqref{eq_Intro_GLS2}. The first one dictates the evolution of the geometry. Due to its structure similar to the mean curvature flow, it is referred to as scaled mean curvature flow equation. The second one is a diffusion equation for the concentration $c$ on the evolving surface $\Gamma$. Hence, in this work, we discuss the coupling of a mean curvature flow-type equation and a diffusion equation, similar as in \cite{PozziStinner1}, \cite{BarrettDeckelnickStyles}, \cite{DeckelnickStyles_FEEA} and \cite{KovacsLiLubich_AConvergentAlgorithmForForcedMCF}. All these contributions are exclusively concerned with numerical analysis. Contrarily, our work contains only analytic results and does not address any numerical approximation. In \cite{AbelsBuergerGarcke}, we proved a short time existence result for \eqref{eq_Intro_GLS}. Now, we investigate several properties of the solution to our system of equations, placing emphasis on whether or not and to what extent the hypersurface in our setting fulfills similar qualitative properties as for the usual mean curvature flow. This depends heavily on the shape of the function $G$: On account of mass conservation, reducing the energy means finding a balance between the two opposing objectives of decreasing the surface area of $\Gamma$ and decreasing the values of $G(c)$. Therefore, we also include a discussion of the physical and mathematical aspects of the shape of the energy density function $G$. Especially, the consequences of the parabolicity conditions are analyzed. These conditions are used in \cite{AbelsBuergerGarcke} for the proof of short time existence and for this reason, we will investigate the properties of solutions mainly under these conditions. \\
The promised properties of a solution to our system of equations are the content of Section \ref{ChapEigLsg}. As a start, we draw a connection between our system of equations and the energy functional: A solution of \eqref{eq_Intro_GLS} never increases the energy and we explain formally how our equations can be interpreted as a gradient flow of \eqref{eq_Intro_Energy}. Next, we show that the parabolically scaled mean curvature flow decreases the surface area just as it is known for the usual mean curvature flow. Still it turns out that, in distinction to the usual flow and even under the parabolicity conditions, an initially sphere-shaped hypersurface does not necessarily collapse in finite time.
Then, the conservation of convexity and mean convexity of the evolving hypersurface is addressed: Whereas the latter is conserved by our scaled mean curvature flow, the former is not - in contrast to the usual mean curvature flow. Afterwards, we discuss the formation of self-intersections which is impossible for the usual mean curvature flow but can occur for our scaled mean curvature flow. Finally, several properties of the concentration function are analyzed. In particular, we prove that a solution of \eqref{eq_Intro_GLS} fulfills mass conservation. \\
Several of these properties are illustrated in the work of Elliott, Garcke and Kovács \cite{ElliottGarckeKovacs} that is concerned with numerical analysis for a generalization of our system of equations. \\

The work at hand is based on the dissertation \cite{Buerger} by the second author.

\subsection{The Scaled Mean Curvature Flow Equation}\label{ChapSMCFeq}

This section is devoted to the scaled mean curvature flow equation \eqref{eq_Intro_GLS1}. If $G \equiv \alpha$ is a constant function then we also have $g(c) = G(c) - G'(c)c = \alpha$. Thus, for $\alpha = 1$, \eqref{eq_Intro_GLS1} is just the usual mean curvature flow equation 
\begin{align*}
V = H
\end{align*}
and for $\alpha \neq 1$ we obtain a constantly scaled equation that behaves equally. \\
We gather some properties of the mean curvature flow. As a wonderful survey article on this topic, we recommend \cite{White}. First of all, surfaces that evolve under mean curvature flow turn smooth instantly. This is due to the parabolicity of the partial differential equation for the motion. But because the equation is non-linear, the general theory for parabolic equations only yields smoothing for a short time and does not forbid later singularities (see \cite{Huisken} for a further discussion). The second property is the decrease of the surface area. In fact, the mean curvature flow turns out to be a gradient flow for the area functional which is explicated e.g. in \cite{GarckeDMV}. With the help of a parabolic maximum principle, one can show that embedded, closed surfaces remain embedded, i.e., do not develop self-intersections, and disjoint, closed surfaces remain disjoint (see \cite[Sections 2.1 and 2.2]{Mantegazza} or \cite[Proposition 2.4]{Ecker}). This implies particularly that closed surfaces have finite lifespans: As explained e.g. in \cite{EckerDMV}, the statement reduces to the simple case of spheres, because any closed surface can be surrounded by a sphere and letting them both evolve under mean curvature flow does not produce any collisions. A further consequence of the parabolic maximum principle is that mean convexity is conserved for closed surfaces (see \cite[Proposition 2.4.1]{Mantegazza}). A surface is called \textit{mean convex} if its mean curvature is non-negative.
Besides these relatively basic properties of the mean curvature flow, one of its most remarkable features is that convex, closed surfaces shrink to round points. This catchy formulation means firstly that any convex, closed surface stays convex and secondly that it becomes asymptotically spherical, i.e., the rescaled surface converges to a sphere. The result was proven in \cite{GageHamilton} for the curve case $d=1$ and in \cite{Huisken} for higher dimensions $d \geq 2$. \\ 
In this work however, we deal with more general, non-constant functions $G$ and therefore end up with a non-constantly scaled version 
\begin{align*}
V = g(c) H
\end{align*}
of the mean curvature flow. Our aim is to analyze which of the properties mentioned above transfer from the usual mean curvature flow to the scaled version. This depends of course heavily on the shape of the function $G$. 

\subsection{The Energy Density Function $G$}\label{ChapG}

Due to its role in the energy functional \eqref{eq_Intro_Energy}, the function $G: \R \rightarrow \R$ is called an energy density function. This section is restricted to the physical relevant case of $c \geq 0$, so it suffices to consider $G: \R_{\geq 0} \rightarrow \R$. Due to entropic effects, it seems natural to assign a high energy to the extreme cases of $c \approx 0$ and $c \approx \infty$ so that their occurrence is avoided. \\
Thus, seeking for a minimal energy, the concentration tends to assume values in the moderate area of $G$. On the other hand, the hypersurface $\Gamma$ tends to minimize its surface area and, on account of mass conservation, thus will force the concentration to grow. The system hence needs to find a balance between these two trends. It is not clear which of them will prevail; whether the hypersurface will shrink to a single point as for the usual mean curvature flow or if the role of the additional concentration is significant enough to stop or at least slow down the shrinking of the surface. Obviously, the sought balance depends on the shape of $G$: A constant energy density function $G$ implies independence of the concentration and results in the usual mean curvature flow. Moderate growth of $G(c)$ for $c \rightarrow \infty$ leads to a weak effect of the concentration and therefore to a behavior still similar to the usual mean curvature flow whereas strong growth of $G(c)$ for $c \rightarrow \infty$ increases the impact of the additional concentration and thus allows for a different geometric behavior. \\
To show the existence of short-time solutions to \eqref{eq_Intro_GLS} in \cite{AbelsBuergerGarcke}, we need a more specific shape of $G$: The equations \eqref{eq_Intro_GLS1} and \eqref{eq_Intro_GLS2} have to be parabolic, so $g>0$ and $G''>0$ have to hold. The second condition $G''>0$ implies that $G$ is convex. This fits nicely to the physical idea of high energy for the extreme cases of $c \approx 0$ and $c \approx \infty$, but it prohibits interesting effects like phase transitions. Hence, we have to restrict to considering only one phase or to describing physical situations that generally allow for convex (Gibbs) energy densities. This is the case, e.g., if we neglect effects of internal energy and simply consider a (convex) mixing entropy density. \\ 
The first condition $g>0$ is more involved and will be discussed now. As $G''>0$, we have 
\begin{align*}
g'(c) = \big( G(c)-G'(c)c \big)' = G'(c) - G''(c)c - G'(c) = -G''(c)c < 0
\end{align*}
for $c \geq 0$. Together, $g>0$ and $g'<0$ imply that $g'(c)$ is bounded for $c \rightarrow \infty$. So, 
\begin{align*}
G''(c) = -\frac{g'(c)}{c}
\end{align*}
tends to zero for $c \rightarrow \infty$ and therefore $G(c)$ can grow at most linearly for $c \rightarrow \infty$. It is even possible to specify the slope of this linear growth: As $G$ is convex, we have
\begin{align*}
G(c) - G'(c_0)c \geq G(c_0) - G'(c_0)c_0 = g(c_0) > 0 
\phantom{x} \Leftrightarrow \phantom{x} 
\frac{G(c)}{c} \geq G'(c_0)
\end{align*}
for all $c,c_0 > 0$ and hence
\begin{align*}
\lim_{c_0 \rightarrow \infty} G'(c_0) \leq \inf_{c>0} \frac{G(c)}{c}.
\end{align*}
The estimate $G'(c_0) \geq G'(1)$ holds for any $c_0 \geq 1$ due to $G''>0$ and so $\lim_{c_0 \rightarrow \infty} G'(c_0) \geq G'(1)$ follows. Altogether, we thus have
\begin{align*}
\lim_{c_0 \rightarrow \infty} |G'(c_0)| \leq \max \left\{ \inf_{c>0} \frac{G(c)}{c}, |G'(1)| \right\}.
\end{align*}
Linear growth definitely is only moderate growth. On that account, we expect the geometric evolution in the parabolically scaled case to vary less from the one of the usual mean curvature flow than it would for arbitrary scaling. Nevertheless, we will investigate the properties of solutions mainly under these parabolicity conditions, as they are required in \cite{AbelsBuergerGarcke} for the proof of short time existence.

\subsection{The Diffusion Equation}\label{ChapDeq}

Now, we discuss the second equation \eqref{eq_Intro_GLS2}. 
If there is no geometric evolution, i.e., if $\Gamma$ is a non-moving hypersurface, the energy functional \eqref{eq_Intro_Energy} is fully determined through $G(c)$ and therefore $G(c)$ can be understood as the free energy of the system. Then, $\mu = G'(c)$ is the chemical potential. Fick's first law $j=-\nabla \mu$ for the flux $j$ and a constant diffusion coefficient, set to one for simplicity, together with the continuity equation $\partial_t c = - \Div j$ finally results in the diffusion equation $\partial_t c = \Delta \mu$. And, indeed, without geometric evolution, i.e. $V=0$, the second equation \eqref{eq_Intro_GLS2} reduces to the simple diffusion equation valid for a concentration independent diffusion coefficient
\begin{align*}
\partial_t c = \Delta_\Gamma \big( G'(c) \big)
\end{align*}
as $\partial^\square c = \partial_t c$ holds for a non-moving hypersurface. \\
In the general case of a non-constant evolving hypersurface, the second equation \eqref{eq_Intro_GLS2} is given by
\begin{align*}
\partial^\square c &= \Delta_\Gamma \big( G'(c) \big) + cHV.
\end{align*}
Again, we have a diffusion equation for the concentration $c$ on the surface $\Gamma$ and the changing of the geometry results in an additional source term that depends linearly on the concentration $c$. \\
We want the concentration $c$ to describe the distribution of a quantity that can neither vanish from nor be added to the surface. The second equation guarantees this conservation of mass (see Theorem \ref{Massenerhaltung}).

\section{Preliminaries}

In the following, let $d,n \in \N_{>0}$ and $k,l \in \N_{\geq 0}$. 

\subsection{Geometric Setting}\label{GeometricSetting}

\begin{definition}[Embedded Hypersurface]\label{hypersurface_Def}$\phantom{x}$ \\
A subset $M \subset \R^{d+1}$ is called a $C^{1+k}$-embedded (closed) hypersurface if 
\begin{enumerate}
\item[(i)] $M$ is a $d$-dimensional $C^{1+k}$-embedded submanifold,
\item[(ii)] there exists a continuous unit normal $\nu_M$, i.e., a continuous vector field $\nu_M: M \rightarrow \R^{d+1}$ with $|\nu_M(p)| = 1$ and $\nu_M(p) \perp T_pM$ for all $p \in M$ and
\item[(iii)] $M$ is connected (and compact) as subset of $\R^{d+1}$.
\end{enumerate}
\end{definition}

In this work, a hypersurface never contains a boundary. A unit normal automatically fulfills $\nu_M \in C^k(M,\R^{d+1})$. 

\begin{definition}[Immersed Hypersurface]$\phantom{x}$ \\
Let $M \subset \R^{d+1}$ be a $C^{1+k}$-embedded (closed) hypersurface and let $\theta: M \rightarrow \R^{d+1}$ be a $C^{1+k}$-immersion, i.e., $\theta \in C^{1+k}(M,\R^{d+1})$ such that its differential $\mathrm{d}_p\theta: T_pM \rightarrow \R^{d+1}$ is injective for all $p \in M$. Then, $\Sigma \defr \theta(M) \subset \R^{d+1}$ is called a $C^{1+k}$-immersed (closed) hypersurface with reference surface $M$ and global parameterization $\theta$. 
\end{definition}

Just as for the embedded case, an immersed hypersurface never contains a boundary. Moreover, we remark that we do not use any topological structure on the immersed hypersurface $\Sigma$ itself but only consider the topology on the (embedded) reference surface $M$. \\
As locally any immersion is an embedding, for every point $p \in M$ there exists an open neighborhood $U \subset M$ such that $\theta(U) \subset \Sigma$ is an \textit{embedded patch}, i.e., an embedded hypersurface, and the restriction $\theta_{|U}$ is an embedding. 
Every locally defined term for embedded hypersurfaces thus can easily be defined also for immersed hypersurfaces, simply defining it on the embedded patches. To avoid confusion in points of self-intersection, we always use the pullback onto the reference surface $M$. 
For example, we define the \textit{surface gradient} of a function $f \in C^1(M,\R)$ by
\begin{align*}
\nabla_\Sigma f_{\, | U} \defr \Big( \nabla_{\theta(U)} \big(f \circ (\theta_{|U})^{-1} \big) \Big) \circ \theta_{|U}.
\end{align*}
Analogously, we define the \textit{Laplace-Beltrami operator} $\Delta_\Sigma$ and in addition, we introduce the \textit{surface Hessian} $D_\Sigma^2$ with 
\begin{align*}
[ D_\Sigma^2 f]_{ij} \defr \big[\nabla_\Sigma \big( [\nabla_\Sigma f]_i \big) \big]_j \text{ for } i,j=1,...,d+1
\end{align*}
for any function $f \in C^2(M,\R)$.

As for every embedded patch $\theta(U)$ the differential $\mathrm{d}_p\theta: T_pM=T_pU \rightarrow T_{\theta(p)}\theta(U)$ is a linear isomorphism, the tangent space of $\Sigma$ at $\theta(p)$ for $p \in M$ is given by $T_p\Sigma \defr \mathrm{d}_p\theta(T_pM)$. Furthermore one can show that orientability transfers from the embedded hypersurface $M$ to the immersed hypersurface $\Sigma = \theta(M)$, meaning that there exists a unit normal $\nu \in C^k(M,\R^{d+1})$ with $|\nu(p)|=1$ and $\nu(p) \perp T_p\Sigma$ for all $p \in M$ (see \cite[Proposition 2.27]{Buerger}). 

\begin{definition}[Mean Curvature]\label{mc_def} $\phantom{x}$ \\
Let $\Sigma=\theta(M)$ be a $C^2$-immersed hypersurface with unit normal $\nu$. Its mean curvature is defined as
\begin{align*}
H \defr -\Div_\Sigma \nu \text{ on } M.
\end{align*}
\end{definition}

Note, that we assign a negative mean curvature to convex surfaces and thus always use the outer unit normal in this work!

\begin{definition}[Evolving Hypersurface]\label{evolvHF_Def} $\phantom{x}$ \\
Let $M \subset \R^{d+1}$ be a $C^{2+k}$-embedded (closed) hypersurface and let $T \in (0,\infty)$. Furthermore, let $\theta: [0,T] \times M \rightarrow \R^{d+1}$ with 
\begin{align*}
\theta \in C^{1+l}\big([0,T],C^{k}(M,\R^{d+1})\big) \cap C^{l}\big([0,T],C^{2+k}(M,\R^{d+1})\big)
\end{align*}
such that $\theta_t \defr \theta(t,\cdot): M \rightarrow \R^{d+1}$ is an embedding / immersion for all $t \in [0,T]$. With $\Gamma(t) \defr \Gamma_t \defr \theta_t(M)$, we call 
\begin{align*}
\Gamma \defr \big\{ \{t\} \times \Gamma(t) \, \big| \, t \in [0,T] \big\}
\end{align*}
a $C^{1+l}$-$\phantom{.}C^{2+k}$-evolving embedded / immersed (closed) hypersurface with reference surface $M$ and global parameterization $\theta$.
\end{definition}

An evolving hypersurface $\Gamma$ is well-defined in the sense that for any time $t$, $\Gamma_t$ is an embedded  or immersed hypersurface again, respectively. In the immersed case, $\Gamma_t$ thus locally is an embedded hypersurface. This locality can even be chosen independently of the time $t$: 
For $U \subset M$ sufficiently small, $\Gamma_{|U} \defr \big\{ \{t\} \times \theta_t(U) \, \big| \, t \in [0,T] \big\}$ is an evolving embedded hypersurface, called an \textit{(embedded) patch of $\Gamma$} (see \cite[Proposition 2.50]{Buerger}). Also, one can show that there exists a so-called normal $\nu \in C^{l}\big([0,T],C^{1+k}(M,\R^{d+1})\big)$ such that $\nu(t,\cdot)$ is a unit normal to $\Gamma(t)$ for all $t \in [0,T]$ (see \cite[Proposition 2.51]{Buerger}).

\begin{definition}[Total and Normal Velocity]\label{velocity_Def} $\phantom{x}$ \\
Let $\Gamma$ be a $C^1$-$\phantom{.}C^2$-evolving immersed hypersurface with global parameterization $\theta$ and normal $\nu$. We define its total velocity
\begin{align*}
V^{\text{tot}} \defr \partial_t \theta
\end{align*}
and its normal velocity 
\begin{align*}
V \defr V^{\text{tot}} \cdot \nu.
\end{align*}
\end{definition}

The normal velocity $V$ is a real number and independent of the global parameterization (see \cite[Remark 24(ii)]{Gpde}).

\begin{definition}[Normal Time Derivative]\label{timederivatives_Def} $\phantom{x}$ \\
Let $\Gamma$ be a $C^1$-$\phantom{.}C^2$-evolving immersed hypersurface with reference surface $M$ and total velocity $V^{\text{tot}}$. Then, for a mapping $f \in C^1([0,T] \times M)$, we define the normal time derivative
\begin{align*}
\partial^\square f \defr \partial_t f - V^{\text{tot}} \cdot \nabla_\Gamma f.
\end{align*}
\end{definition}

The normal time derivative is independent of the global parameterization (see \cite[Remark 29(iii) and (iv)]{Gpde}).

\subsection{Short time existence}

In \cite{AbelsBuergerGarcke} we proved short time existence for a parameterized version of our system \eqref{eq_Intro_GLS}. The parameterization relies on the fact that due to a single codimension, we can characterize the evolution of any surface simply by its movement in normal direction. Thus, fixing a reference surface $M$ and a so-called height-function $\rho: [0,T] \times M \rightarrow \R$ completely determines an evolving surface, as the following lemma states (see e.g. \cite[Corollary 2.63]{Buerger} for a proof).

\begin{lemma}
Let $\Sigma = \bar{\theta}(M) \subset \R^{d+1}$ be a $C^{3+k}$-embedded / immersed closed hypersurface with unit normal $\nu_\Sigma$ and let $T \in (0,\infty)$. Furthermore, let 
$\rho \in C^{1+l}\big([0,T],C^{k}(M)\big) \cap C^{l}\big([0,T],C^{2+k}(M)\big)$ with $\|\rho\|_{C^0([0,T] \times M)}$ sufficiently small. We define 
\begin{align*}
\theta_\rho: [0,T] \times M \rightarrow \R^{d+1}, \phantom{x} \theta_\rho(t,p) \defr \bar{\theta}(p) + \rho(t,p)\nu_\Sigma(p).
\end{align*}
Then, with $\Gamma_\rho(t) \defr \theta_\rho(t,M)$,
\begin{align*}
\Gamma_\rho \defr \big\{ \{t\} \times \Gamma_\rho(t) \, \big| \, t \in [0,T] \big\}
\end{align*}
defines a $C^{1+l}$-$\phantom{.}C^{2+k}$-evolving embedded / immersed hypersurface with reference surface $M$ and global parameterization $\theta_\rho$. 
\end{lemma}

Note, that we forgo the Hölder setting used in \cite{AbelsBuergerGarcke} for simplicity. Instead, we formulate the estimate resulting from the properties of the Hölder functions explicitely in Theorem \ref{lokEx_param}. With 
\begin{align*}
H(\rho) \defr H_{\Gamma_\rho} \phantom{xxx} \text{ and } \phantom{xxx}
a(\rho) \defr \frac{1}{\nu_\Sigma \cdot \nu_{\Gamma_\rho}},
\end{align*}
the short time existence result reads as follows (see \cite[Theorem 3.28]{AbelsBuergerGarcke}).

\begin{theorem}\label{lokEx_param}
Let $G \in C^7(\R)$ with $G''>0$ and $g \defr G-G' \cdot \Id > 0$. Moreover, let $\Sigma=\bar{\theta}(M)$ be a $C^5$-immersed closed hypersurface with unit normal $\nu_\Sigma$. Let $u_0 \in C^4(M)$ and $\delta_1>0$ be arbitrary. Then, choose $\delta_0 = \delta_0(\Sigma,u_0,\delta_1)>0$ and $T = T(\Sigma,u_0,\delta_1)>0$ sufficiently small. For every function $\rho_0 \in C^4(M)$ with $\|\rho_0\|_{C^4(M)} < \delta_1$ and $\|\rho_0\|_{C^2(M)} < \delta_0$, there exists a solution $(\rho,u)$ with $\rho, u \in C^1\big([0,T],C^0(M)\big) \cap C^0\big([0,T],C^2(M)\big)$ to
\begin{align*}
\left\{
\begin{aligned}
\partial_t \rho \phantom{bl} &= \phantom{bl} g(u)a(\rho)H(\rho) & &\text{ on } [0,T] \times M, \\
\partial_t u \phantom{bl} &= \phantom{bl} \Delta_{\Gamma_\rho} G'(u) + g(u) a(\rho) H(\rho) \nu_\Sigma \cdot \nabla_{\Gamma_\rho} u + g(u) H(\rho)^2 u & &\text{ on } [0,T] \times M, \\
\rho(0) \phantom{bl} &= \phantom{bl} \rho_0 & &\text{ on } M, \\
u(0) \phantom{bl} &= \phantom{bl} u_0 & &\text{ on } M.
\end{aligned}
\right.
\end{align*}
Furthermore, there exists a constant $R = R(\Sigma,u_0,\delta_1)>0$ independent of $\rho_0$ with 
\begin{align*}
\|\partial_t \rho(t)-\partial_t \rho(0)\|_{C^0(M)} + \|\rho(t)-\rho(0)\|_{C^2(M)} &\leq RT^{1/4} \phantom{xx} \text{and} \\
\|\partial_t u(t)-\partial_t u(0)\|_{C^0(M)} + \|u(t)-u(0)\|_{C^2(M)} &\leq RT^{1/4}
\end{align*}
for all $t \in [0,T]$. For any two solutions, there exists $\overline{T} \in (0,T]$ such that the solutions coincide on $[0,\overline{T}]$.
\end{theorem}

In particular, the result is formulated for the case of immersed hypersurfaces and yields a uniform lower bound
on the existence time that allows for small changes in the initial value of the height function. This is crucial for our proof of the existence of self-intersections in Section \ref{ChapSI}.

\subsection{Maximum Principle}

A matrix valued function $A: \Omega \rightarrow \R^{n \times n}$, $\Omega \subset \R^d$, is called \textit{positive definite} on $\Omega$ if 
\begin{align*}
\xi^\top A(x) \xi > 0
\end{align*}
holds for every $x \in \Omega$ and $\xi \in \R^n \setminus \{0\}$.
For two matrices $A,\widetilde{A} \in \R^{n \times n}$, we define
\begin{align*}
A \colon \widetilde{A} \defr \sum_{i,j=1}^n A_{ij} \widetilde{A}_{ij}.
\end{align*}

To prove conservation properties in Section \ref{ChapEigLsg}, we want to apply a maximum principle for functions defined on evolving closed hypersurfaces in the following setting:

\begin{ass}\label{ass_maxprinzip}
Let $\Gamma$ be a $C^1$-$\phantom{.}C^2$-evolving immersed closed hypersurface with reference surface $M \subset \R^{d+1}$ and global parameterization $\theta$. Furthermore, let $A: [0,T] \times M \rightarrow \R^{(d+1) \times (d+1)}$, $B: [0,T] \times M \rightarrow \R^{d+1}$ and $C: [0,T] \times M \rightarrow \R$ be continuous with $A$ symmetric and positive definite on $[0,T] \times M$. Let $w \in C^1\big([0,T],C^0(M)\big) \cap C^0\big([0,T],C^2(M)\big)$ and define
\begin{align*}
\mathcal{L}w \defr - \partial^\square w + A \colon D_\Gamma^2 w + B \cdot \nabla_\Gamma w + Cw.
\end{align*}
\end{ass}

The maximum principle we will use is proven in \cite[Corollary 2.148]{Buerger} and reads as follows.

\begin{proposition}[Maximum Principle on Evolving Closed Hypersurfaces]\label{MaxPrinzip_Folg} $\phantom{x}$ \\
Let Assumptions \ref{ass_maxprinzip} hold true with $\mathcal{L}w \geq 0$ on $[0,T] \times M$ and $w(0,\cdot) \leq 0$ on $M$. Then, we have $w(t,\cdot) \leq 0$ for all $t \in [0,T]$ and there exists $t_0 \in [0,T]$ with
\begin{align*}
w(t,\cdot) \equiv 0 \text{ for all } t \in (0,t_0]
\phantom{xx} \text{ and } \phantom{xx}
w(t,\cdot) < 0 \text{ for all } t \in (t_0,T].
\end{align*}
\end{proposition}

\section{Properties of Solutions}\label{ChapEigLsg}

In the following, we will investigate properties of solutions to our equations \eqref{eq_Intro_GLS}, placing emphasis on whether or not and to what extent the surface in our setting evolves with similar properties as for the usual mean curvature flow. Universal properties and properties coinciding with the usual mean curvature flow will be analyzed in the following setting.

\begin{ass}\label{ass_concentration}
Let $G \in C^3(\R)$ and define $g \defr G - G' \cdot \Id$. \\
Moreover, let $\Gamma$ be a \mbox{$C^1$-$\phantom{.}C^2$-}evolving immersed closed hypersurface with reference surface $M \subset \R^{d+1}$ and let the function \mbox{$c \in C^1\big([0,T],C^0(M)\big) \cap C^0\big([0,T],C^2(M)\big)$} be such that the pair $(\Gamma,c)$ is a solution to the system \eqref{eq_Intro_GLS}
\begin{subequations}
\begin{align*}
V &= g(c) H, \\
\partial^\square c &= \Delta_\Gamma \big( G'(c) \big) + cHV. 
\end{align*}
\end{subequations}
\end{ass}

Furthermore, we will construct explicit counterexamples to show a behavior differing from the usual mean curvature flow. Because they are used in \cite{AbelsBuergerGarcke} for the proof of short time existence of \eqref{eq_Intro_GLS}, we will focus on energy densities $G$ that fulfill the parabolicity conditions 
\begin{align*}
G''>0 \phantom{xx} \text{and} \phantom{xx} g>0.
\end{align*}
With the condition $g>0$, we call \eqref{eq_Intro_GLS1} a \textit{parabolically scaled mean curvature flow}. In order to apply the short time existence result (Theorem \ref{lokEx_param}), we also require $G \in C^7(\R)$.

\subsection{Gradient Flow of the Energy Functional}\label{Sec_GradientFlow}

In this section, we establish a connection between the equations \eqref{eq_Intro_GLS} and the energy functional \eqref{eq_Intro_Energy}: The evolution of a pair $(\Gamma,c)$ that decreases the energy \eqref{eq_Intro_Energy} most efficiently is characterized by our equations \eqref{eq_Intro_GLS}. In other words, the system of equations is a gradient flow of the energy functional. \\
First, we state that a solution of \eqref{eq_Intro_GLS} can never increase the energy \eqref{eq_Intro_Energy}.

\begin{theorem}[Non-Increase of Energy]\label{energy_decrease} $\phantom{x}$ \\
Suppose Assumptions \ref{ass_concentration} are valid. Then, the energy $t \mapsto \mathrm{E}\big(\Gamma(t),c(t)\big)$ is non-increasing. 
\end{theorem}

\begin{proof}
With the help of the transport theorem (Proposition \ref{transporttheorem}) and Gauß' theorem on closed hypersurfaces (Proposition \ref{partialintegration}), we get 
\begin{align*}
\frac{\mathrm{d}}{\mathrm{d}t} \mathrm{E}(\Gamma,c)
&= \int_{\Gamma(t)} G'(c) \partial^\square c - G(c)HV \, \mathrm{d}\mathcal{H}^d 
= \int_{\Gamma(t)} G'(c) \Delta_\Gamma \big(G'(c)\big) -g(c)HV \, \mathrm{d}\mathcal{H}^d \\
&= - \int_{\Gamma(t)} \big| \nabla_\Gamma G'(c) \big|^2 + V^2 \, \mathrm{d}\mathcal{H}^d
\leq 0. \qedhere
\end{align*}
\end{proof}

Hence, a solution $(\Gamma,c)$ of \eqref{eq_Intro_GLS} can never increase the energy functional $\mathrm{E}$. As long as the geometry of the system changes, i.e. $V \neq 0$, the energy will actually decrease. Also, assuming $G'$ not to be constant, a non-uniform distribution of the quantity described by $c$ in general also leads to an actual decrease of the energy. Note that due to $V = g(c)H$ by \eqref{eq_Intro_GLS1} and because a closed hypersurface $\Gamma(t)$ cannot have vanishing mean curvature $H$ everywhere (Proposition \ref{mc_zero}), the parabolicity condition $g>0$ assumed in \cite{AbelsBuergerGarcke} implies an actual decrease of the energy. \\
In the following, we explain formally how the equations \eqref{eq_Intro_GLS} can be seen as a gradient flow of the energy functional \eqref{eq_Intro_Energy}. This means that a solution $(\Gamma,c)$ of \eqref{eq_Intro_GLS} even decreases the energy functional in an optimal way. 
The techniques used in this section are based on corresponding arguments in \cite{GarckeDMV} for the usual mean curvature flow of evolving hypersurfaces. We extend these considerations to our setting with the additional concentration that describes a distribution on the surface. A solution of \eqref{eq_Intro_GLS} satisfies mass conservation (see Theorem \ref{Massenerhaltung}). So, for a constant mass $m \in \R_{>0}$, we consider the set
\begin{align*}
M^m &\defr \left\{ (\Sigma,c) \, \big| \, \Sigma \subset \R^{d+1} \text{ smooth, closed, embedded hypersurface}, \phantom{\int_\Sigma} \right. \\
&\phantom{\defr (\Sigma,c) \, \big| \, x } \left. c: \Sigma \rightarrow \R \text{ smooth concentration with } \int_\Sigma c \, \mathrm{d}\mathcal{H}^d = m \right\} 
\end{align*}
of all surfaces $\Sigma$ in $\R^{d+1}$ and concentrations $c: \Sigma \rightarrow \R$ such that the total mass of the quantity, whose distribution on $\Sigma$ is described by $c$, equals $m$. In a formal way, we endow $M^m$ with the tangent space
\begin{align*}
T_{(\Sigma,c)}M^m = \left\{ \big(V,w\big) \, \big| \, V,w: \Sigma \rightarrow \R \text{ smooth with } \int_\Sigma w-cHV \, \mathrm{d}\mathcal{H}^d = 0 \right\}. 
\end{align*}
Here, $V$ is a possible normal velocity of $\Sigma$, $w$ is a variation of the concentration and the additional condition $\int_\Sigma w-cHV = 0$ with mean curvature $H$ of $\Sigma$ ensures that the change is such that mass conservation holds. We now elaborate how such a pair $(V,w)$ of smooth functions $V,w: \Sigma \rightarrow \R$ arises as the differential of a map in $M^m$, hence as a ``tangent vector'' of $M^m$ in a point $(\Sigma,c)$. \\
Let $\theta_t: \Sigma \rightarrow \R^{d+1}, \, \theta_t(p) \defr p + tV(p)\nu(p)$ with the smooth unit normal $\nu$ of $\Sigma$. According to \cite[Proposition 2.60]{Buerger}, for $|t|$ sufficiently small, $\theta_t$ is a smooth diffeomorphism onto its image $\Gamma_t \defr \theta_t(\Sigma)$ which again is a smooth, closed, embedded hypersurface in $\R^{d+1}$. In particular, $\big\{ \{t\} \times \Gamma_t \, \big| \, t \in (-\varepsilon,\varepsilon) \big\}$ is an evolving hypersurface as in Definition \ref{evolvHF_Def} and its normal velocity in $t=0$ equals $V$. So, $V$ really is the normal velocity of $\Sigma=\Gamma_0$ as already stated above. \\ 
Moreover, with
\begin{align*}
m_t \defr \left( \int_{\Gamma_t} 1 \, \mathrm{d}\mathcal{H}^d \right)^{-1} \left( m - \int_{\Gamma_t} (c+tw) \circ \theta_t^{-1} \, \mathrm{d}\mathcal{H}^d \right) \in \R,
\end{align*}
$c_t \defr (c+tw) \circ \theta_t^{-1} + m_t: \Gamma_t \rightarrow \R$ defines a smooth function on $\Gamma_t$. We consider $\eta(t) \defr (\Gamma_t,c_t)$ for $t \in (-\varepsilon,\varepsilon)$ and claim that $\eta$ is a map in $M^m$ through $(\Sigma,c)$ with differential $(V,w)$. \\
By construction, we have
\begin{align*}
\int_{\Gamma_t} c_t \, \mathrm{d}\mathcal{H}^d
&= \int_{\Gamma_t} (c+tw) \circ \theta_t^{-1} + m_t \, \mathrm{d}\mathcal{H}^d \\
&= \int_{\Gamma_t} (c+tw) \circ \theta_t^{-1} \, \mathrm{d}\mathcal{H}^d + \left( m - \int_{\Gamma_t} (c+tw) \circ \theta_t^{-1} \mathrm{d}\mathcal{H}^d \right) 
= m
\end{align*}
for every $t \in (-\varepsilon,\varepsilon)$, such that $\eta(t) \in M^m$ holds for every $t \in (-\varepsilon,\varepsilon)$. Furthermore, we have $\eta(0)=(\Sigma,c)$, since
\begin{align}\label{eq_EQ_normaltd1}
\int_{\Gamma_t} (c+tw) \circ \theta_t^{-1} \, \mathrm{d}\mathcal{H}^d_{\phantom{x}|t=0}
= \int_\Sigma c \, \mathrm{d}\mathcal{H}^d
= m
\end{align}
implies $m_0=0$. In addition, 
\begin{align*}
\partial^\square \big( (c+tw) \circ \theta_t^{-1} \big)_{|t=0}
= \frac{\mathrm{d}}{\mathrm{d}t} (c+tw)_{|t=0} - \partial_t \theta_t \cdot \nabla_{\Gamma_t} \big((c+tw) \circ \theta_t^{-1} \big)_{|t=0}
= w - V\nu \cdot \nabla_\Sigma c
= w
\end{align*}
holds because as an element of the tangent space of $\Sigma$, the surface gradient $\nabla_\Sigma c$ is perpendicular to the normal $\nu$. With the transport theorem (Proposition \ref{transporttheorem}) and the additional condition for $(V,w)$,
\begin{align}\label{eq_EQ_normaltd2}
\frac{\mathrm{d}}{\mathrm{d}t}_{|t=0} \int_{\Gamma_t} (c+tw) \circ \theta_t^{-1} \, \mathrm{d}\mathcal{H}^d
&= \int_\Sigma \partial^\square \big((c+tw) \circ \theta_t^{-1}\big) - \big((c+tw)\circ \theta_t^{-1}\big)_{|t=0} HV \, \mathrm{d}\mathcal{H}^d \notag \\
&= \int_\Sigma w - cHV \, \mathrm{d}\mathcal{H}^d
= 0
\end{align}
follows. Equations \eqref{eq_EQ_normaltd1} and \eqref{eq_EQ_normaltd2} yield $\partial^\square m_{t \, |t=0} = \frac{\mathrm{d}}{\mathrm{d}t} m_{t \, |t=0} = 0$,
such that we finally have
\begin{align*}
\partial^\square c_{t \, |t=0} 
= \partial^\square \big( (c+tw) \circ \theta_t^{-1} \big)_{|t=0} + \partial^\square m_{t \, |t=0}
= w.
\end{align*}
As the normal velocity of the evolving hypersurface $\big\{ \{t\} \times \Gamma_t \, \big| \, t \in (-\varepsilon,\varepsilon) \big\}$ in $t=0$ is $V$ and the normal time derivative of $c_t$ in $t=0$ is $w$, the pair $(V,w)$ fully determines the differential $\eta'(0)$. In particular, $(V,w)$ can be interpreted as ``tangential vector'' to $M^m$ in $(\Sigma,c)$. \\
On the tangent space $T_{(\Sigma,c)}M^m$ we define an $L^2$-$H^{-1}$-inner product
\begin{align*}
&\big\langle (V_1,w_1),(V_2,w_2) \big\rangle 
\defr \int_\Sigma V_1 V_2 \, \mathrm{d}\mathcal{H}^d + \int_\Sigma \nabla_\Sigma u_1 \cdot \nabla_\Sigma u_2 \, \mathrm{d}\mathcal{H}^d \\
&= \int_\Sigma V_1 V_2 \, \mathrm{d}\mathcal{H}^d + \int_\Sigma u_1 \big(w_2-cHV_2\big) \, \mathrm{d}\mathcal{H}^d 
= \int_\Sigma \big(V_1-u_1cH\big) V_2 \, \mathrm{d}\mathcal{H}^d + \int_\Sigma u_1w_2 \, \mathrm{d}\mathcal{H}^d,
\end{align*}
where $-\Delta_\Sigma u_i = w_i - cHV_i$ holds on $\Sigma$. These functions $u_i$ are well-defined, as due to Gauß' theorem on closed hypersurfaces (Proposition \ref{partialintegration}), $\int_\Sigma w_i-cHV_i \, \mathrm{d}\mathcal{H}^d = \int_\Sigma -\Delta_\Sigma u_i \, \mathrm{d}\mathcal{H}^d = 0$ is fulfilled for $(V_i,w_i) \in T_{(\Sigma,c)}M^m$. Using a Lax-Milgram type argument, one can show that this is exactly the solvability condition for $-\Delta_\Sigma u_i = w_i - cHV_i$ on $\Sigma$.
The choice of an $H^{-1}$-inner product for the concentration part ensures the conservation of mass and the $L^2$-inner product for the surface part results in decreasing surface area just as for the usual mean curvature flow. \\
Now, we want to identify the gradient flow of the energy functional $\mathrm{E}$ (see \eqref{eq_Intro_Energy}) with respect to this inner product. With the help of the transport theorem (Proposition \ref{transporttheorem}), the total differential of $\mathrm{E}$ in $(\Sigma,c)$ in direction of $(V,w) \in T_{(\Sigma,c)}M^m$ is given by
\begin{align*}
\mathrm{DE}(\Sigma,c)(V,w)
&= \frac{\mathrm{d}}{\mathrm{d}t} \mathrm{E}\big( \Gamma_t,c_t \big)_{\, | t=0}
= \frac{\mathrm{d}}{\mathrm{d}t} \int_{\Gamma_t} G(c_t) \, \mathrm{d}\mathcal{H}^d_{\, |t=0} \\
&= \int_{\Sigma} \partial^\square \big( G(c_t) \big)_{\, |t=0} - G(c)HV \, \mathrm{d}\mathcal{H}^d 
= \int_\Sigma G'(c) w - G(c)HV \, \mathrm{d}\mathcal{H}^d
\end{align*}
with $\Gamma_t$ and $c_t$ defined as before. If we choose $(V_g,w_g) \defr \text{grad } \mathrm{E}(\Sigma,c) \in T_{(\Sigma,c)}M^m$ as a notation for the gradient of $\mathrm{E}$ in $(\Sigma,c)$, then for any direction $(V,w) \in T_{(\Sigma,c)}M^m$,
\begin{align*}
\big\langle \text{grad } \mathrm{E}(\Sigma,c) , (V,w) \big\rangle
= \big\langle (V_g,w_g),(V,w) \big\rangle
= \int_\Sigma \big( V_g - u_g c H \big) V \, \mathrm{d}\mathcal{H}^d + \int_\Sigma u_g w \, \mathrm{d}\mathcal{H}^d
\end{align*}
holds with $-\Delta_\Sigma u_g = w_g - cHV_g$ as before. Since the gradient is defined through
\begin{align*}
\mathrm{DE}(\Sigma,c)(V,w)
= \big\langle \text{grad } \mathrm{E}(\Sigma,c) , (V,w) \big\rangle
\end{align*}
for all $(V,w) \in T_{(\Sigma,c)}M^m$, we obtain $u_g = G'(c)$ and
\begin{align*}
V_g &= u_g c H - G(c)H = \big(G'(c)c - G(c)\big) H, \\
w_g &= - \Delta_\Sigma u_g + cHV_g = -\Delta_\Sigma \big( G'(c) \big) + cHV_g.
\end{align*} 
As explained above, the differential $\frac{\mathrm{d}}{\mathrm{d}t} (\Gamma_t,c_t)_{\, |t=0}$ is determined by the normal velocity $V$ of the evolving hypersurface $\big\{ \{t\} \times \Gamma_t \, \big| \, t \in (-\varepsilon,\varepsilon) \big\}$ in $t=0$ and the normal time derivative $\partial^\square c$ of $c_t$ in $t=0$. The family $(\Gamma_t,c_t)_{t \in (-\varepsilon,\varepsilon)}$ hence is a solution to the desired gradient flow in $t=0$ if and only if
\begin{align*}
(V,\partial^\square c) = -\text{grad } \mathrm{E}(\Sigma,c) = -(V_g,w_g)
\end{align*}
is valid. So, the gradient flow of the energy functional $\mathrm{E}$ with respect to our $L^2-H^{-1}$-inner product is the system \eqref{eq_Intro_GLS} as claimed.

\subsection{Decrease of Surface Area and Infinite Existence Time}

A simple calculation shows a first analogy between the usual and the scaled mean curvature flow: The strict decrease of surface area known from the usual mean curvature flow also holds for the parabolically scaled one. In particular, stationary solutions are not possible. Without the parabolicity condition, a different behavior of the surface would be possible of course. 

\begin{theorem}[Decrease of Surface Area]\label{surfacearea} $\phantom{x}$ \\
Suppose Assumptions \ref{ass_concentration} are valid with $g>0$. Then, the surface area of $\Gamma$ is strictly decreasing.
\end{theorem}

\begin{proof}
The transport theorem (Proposition \ref{transporttheorem}) together with \eqref{eq_Intro_GLS1} yields
\begin{align*}
\frac{\mathrm{d}}{\mathrm{d}t} \int_{\Gamma} 1 \, \mathrm{d}\mathcal{H}^d
= - \int_{\Gamma} HV \, \mathrm{d}\mathcal{H}^d
= - \int_{\Gamma} g(c)H^2 \, \mathrm{d}\mathcal{H}^d.
\end{align*}
Due to $g>0$ and because a closed surface can not have vanishing mean curvature $H$ everywhere (Proposition \ref{mc_zero}), the surface area of $\Gamma$ is strictly decreasing.
\end{proof}

To get a further impression of the evolution prescribed by our system of equations \eqref{eq_Intro_GLS}, we study the radial symmetric case as in \cite{Garcke}. The usual mean curvature flow forces a convex, closed surface to shrink to a round point in finite time. Obviously, if the surface in our example does shrink to a single point, it also shrinks to a round point just as for the usual mean curvature flow; but roundness is true anyway due to the construction of the surface as radial symmetric sphere. So, we focus on the collapsing of the surface and whether this happens in finite time. \\
A radial symmetric situation means that the hypersurface $\Gamma(t) = \partial B_{R(t)}(0) \subset \R^{d+1}$ is a sphere and the concentration $c(t): \Gamma(t) \rightarrow \R$ is constant in space. As mass conservation is fulfilled (see Theorem \ref{Massenerhaltung}),
\begin{align*}
m = \int_{\Gamma(t)} c(t) \, \mathrm{d}\mathcal{H}^d
= c(t) \mathcal{H}^d\big( \Gamma(t) \big)
= \alpha_d c(t) R(t)^d 
\phantom{xx} \Leftrightarrow \phantom{xx} 
c(t) = \frac{m}{\alpha_d} R(t)^{-d}
\end{align*}
holds with a constant $\alpha_d$ only depending on the dimension $d$. So, the whole evolution of the pair $\big(\Gamma(t),c(t)\big)$ is characterized by the function $R: [0,T) \rightarrow (0,\infty)$. Due to our sign convention, we have
\begin{align*}
\nu(t,p) = \frac{p}{R(t)}
\phantom{xx} \text{ and } \phantom{xx}
H(t,p) = - \frac{d}{R(t)}
\end{align*}  
for the unit normal and the mean curvature of $\Gamma(t)$ in $p \in \Gamma(t)$. As 
\begin{align*}
\theta: [0,T] \times \Gamma_0 \rightarrow \R^{d+1}, \phantom{xx} \theta(t,z) \defr \frac{R(t)}{R_0}z
\end{align*}
defines a global parameterization for the evolving hypersurface, the normal velocity of $\Gamma(t)$ in $p=\theta(t,z) \in \Gamma(t)$ is given by
\begin{align*}
V(t,p) = \partial_t \theta(t,z) \cdot \nu(t,p)
= \frac{R'(t)}{R_0} z \cdot \frac{p}{R(t)}
= \frac{R'(t)}{R_0} z \cdot \frac{z}{R_0}
= R'(t).
\end{align*}
The normal time derivative of the constant-in-space function $c(t)$ in $p=\theta(t,z) \in \Gamma(t)$ is
\begin{align*}
\partial^\square c(t,p) 
= \partial^\circ c(t,p)
= \frac{\mathrm{d}}{\mathrm{d}t} c(t)
= -\frac{md}{\alpha_d} R(t)^{-d-1} R'(t)
\end{align*}
and thus the concrete form of the second equation \eqref{eq_Intro_GLS2}
\begin{align*}
\partial^\square c 
= -\frac{md}{\alpha_d} R^{-d-1} R' 
= 0 + \frac{m}{\alpha_d} R^{-d} \cdot \frac{-d}{R} \cdot R'
= \Delta_\Gamma \big( G'(c) \big) + cHV
\end{align*}
is automatically fulfilled. This was to be expected because the diffusion equation for spatially constant functions $c$ reduces to $\partial^\square c = cHV$. As also the geometric quantities $H$ and $V$ are spatially constant, this reduced equation is equivalent to mass conservation (cf. Theorem \ref{Massenerhaltung}) and we chose the time dependence of $c$ such that mass conservation is fulfilled. 
The impact of the additional concentration hence is limited to the non-constant scaling factor $g(c)$ in the mean curvature flow equation. Thus, the question we seek to answer is whether there exists a shape of $g$ such that the evolution of the geometry differs from the one for a constant function $g$, i.e., the usual mean curvature flow. \\
We now turn to the first equation which, in the radial symmetric case, transforms to
\begin{align*}
R' = V = g(c) H = - g\left( \frac{m}{\alpha_d} R^{-d} \right) \frac{d}{R} \defl f(R).
\end{align*}
Assuming the parabolicity condition $g>0$, we have $f<0$ on $(0,\infty)$ and therefore we can apply separation of variables which yields
\begin{align*}
t 
= \int_0^t \frac{R'(s)}{f(R(s))} \, \mathrm{d}s
= \int_{R_0}^{R(t)} \frac{1}{f(z)} \, \mathrm{d}z
\defl F\big(R(t)\big).
\end{align*}
Due to $f<0$ on $(0,\infty)$, $F: (0,\infty) \to \R$ is strictly decreasing and thus bijective onto its image. Therefore we can define $R(t) \defr F^{-1}(t)$ for all $t \in F\big((0,\infty)\big)$; but as only positive times $t>0$ are considered, we reduce to $t \in (0,T) \defr (0,\infty) \cap F\big((0,\infty)\big)$. We have
\begin{align}\label{T=F(0)}
\lim_{R \searrow 0} F(R) &= \int_{R_0}^0 \frac{1}{f(z)} \, \mathrm{d}z = - \int_0^{R_0} \frac{1}{f(z)} \, \mathrm{d}z > 0, \\
\lim_{R \nearrow \infty} F(R)  &= \int_{R_0}^{\infty} \frac{1}{f(z)} \, \mathrm{d}z < 0 \nonumber
\end{align}
and thus $(0,T) = \big(0,F(0)\big)$. In particular, 
\begin{align*}
R(t) = F^{-1}(t) \rightarrow F^{-1}(T) = F^{-1}\big(F(0)\big) = 0
\end{align*}
holds for $t \rightarrow T$. So, just as for the usual mean curvature flow, the hypersurface shrinks to a single point under the scaled mean curvature flow with $g>0$. Due to the assumed parabolicity condition $g>0$, the strict decrease of the surface area was already known from Theorem \ref{surfacearea}. 
But, in contrast to the usual mean curvature flow, even for $g>0$, the final time $T$ at which the surface collapses does not need to be finite for the scaled mean curvature flow. We have (see \eqref{T=F(0)})
\begin{align}\label{T}
T = F(0) = - \int_0^{R_0} \frac{1}{f(z)} \, \mathrm{d}z = \frac{1}{d} \int_0^{R_0} \frac{z}{g\left( \frac{m}{\alpha_d}z^{-d} \right)} \, \mathrm{d}z
\end{align}
and for a suitable choice of $g$, this turns out to be $T = \infty$. 
Choose for example the power function
\begin{align*}
G(c) \defr \frac{1}{1-s} c^s + \alpha c
\end{align*}
with $s < -\frac{2}{d}$ and $\alpha > 0$ satisfying $G(c) \rightarrow \infty$ for $c \rightarrow \infty$. It follows that
\begin{align*}
G''(c) = -s c^{s-2}
\phantom{xx} \text{ and } \phantom{xx}
g(c) = G(c)-G'(c)c = c^s
\end{align*}
and thus both parabolicity conditions $G''>0$ and $g>0$ are fulfilled for the physical relevant case $c>0$.

\begin{theorem}[Infinite Existence Time] $\phantom{x}$ \\
Let the energy density $G$, the evolving surface $\Gamma$ and the concentration $c$ be as above. The evolution of $(\Gamma,c)$ under \eqref{eq_Intro_GLS} exists for all times.
\end{theorem}

\begin{proof}
Calculation of the final time $T$ (cf. \eqref{T}) yields
\begin{align*}
T 
= \frac{1}{d} \int_0^{R_0} \frac{z}{g\left( \frac{m}{\alpha_d}z^{-d} \right)} \, \mathrm{d}z 
= C(m,d,s) \int_0^{R_0} z^{sd+1} \, \mathrm{d}z
= \frac{C(m,d,s)}{sd+2} \, z^{sd+2} \Big|_0^{R_0}
\end{align*}
with a positive constant $C(m,d,s)$ depending on the mass $m$, the dimension $d$ and the power variable $s$. Due to the choice $s < -\frac{2}{d} \Leftrightarrow sd+2<0$, this adds up to $T=\infty$.
\end{proof}

We have seen that, assuming the parabolicity condition $g>0$, a surface will always shrink and thus evolve very similar as for the usual mean curvature flow. But even a positive scaling factor $g>0$ can slow down this behavior enforced by the mean curvature flow and prevent the surface from collapsing, at least in finite time. The additional concentration thus does have an effect on the qualitative geometric evolution of the surface in our system \eqref{eq_Intro_GLS}.

\subsection{Conservation of Mean Convexity}\label{Chap_MeanConvex}

A solution of $\eqref{eq_Intro_GLS}$ conserves its mean convexity: If the initial hypersurface $\Gamma_0$ is mean convex, i.e., $H(0,\cdot) \geq 0$, then the evolving hypersurface remains mean convex for all further times, i.e., we have $H(t,\cdot) \geq 0$ for all $t>0$. 
To show this, we want to apply maximum principles to $w=-H$. We will observe later, see Remark \ref{MeanConvex_Lemma_skalierterMCF}, that maximum principles without (sign) conditions on the zeroth-order term are necessary to deduce the conservation of mean convexity. Hence, the typical weak maximum principles as in \cite[Proposition 3.1]{Ecker} and \cite[§7.1: Theorem 8 and Theorem 9]{Evans} (or rather their transfer to hypersurfaces) can not be used. Instead, we derived a suitable weak maximum principle in \cite[Section 2.4.1]{Buerger}. A strong maximum principle without assumptions on the zeroth-order term can be found in the literature for domains in $\R^d$ (see \cite[Theorem 4.26]{RenardyRogers}) and is transferred to hypersurfaces in \cite[Section 2.4.2]{Buerger}. These two maximum principles are combined in Proposition \ref{MaxPrinzip_Folg}.\\

We prove the conservation of mean convexity not only for hypersurfaces that evolve by the scaled mean curvature flow but also for more general evolving hypersurfaces that satisfy the following assumptions: 

\begin{ass}\label{ass_meanconvexity}
Let $\Gamma$ be a $(C^2$-$\phantom{.}C^2) \cap (C^1$-$\phantom{.}C^4)$-evolving immersed closed hypersurface with reference surface $M \subset \R^{d+1}$, normal $\nu$ and mean curvature $H$ 
that evolves with normal velocity $V=V(H)$. Furthermore, we assume
\begin{align*}
\Delta_\Gamma V + V \big| \nabla_\Gamma \nu \big|^2
= A \colon D_\Gamma^2 H + B \cdot \nabla_\Gamma H + CH \quad \text{ on } [0,T] \times M
\end{align*}
with continuous $A: [0,T] \times M \rightarrow \R^{(d+1) \times (d+1)}$, $B: [0,T] \times M \rightarrow \R^{d+1}$ and $C: [0,T] \times M \rightarrow \R$ such that $A$ is symmetric and positive definite on $[0,T] \times M$. 
\end{ass}
With Proposition \ref{mc_formulas}, we have 
\begin{align*}
\partial^\square H = \Delta_\Gamma V + V \big| \nabla_\Gamma \nu \big|^2.
\end{align*}
The regularity of the evolving hypersurface guarantees in particular that all occuring derivatives of the mean curvature are well-defined as Definition \ref{mc_def} implies
\begin{align*}
H \in C^1\big([0,T],C^0(M)\big) \cap C^0\big([0,T],C^2(M)\big).
\end{align*}
Under these conditions, mean convexity of the hypersurface is conserved. Even more, the mean curvature instantly turns strictly positive.

\begin{theorem}[Conservation of Mean Convexity]\label{MeanConvex_H} $\phantom{x}$ \\
We suppose Assumptions \ref{ass_meanconvexity} are valid with $H(0) \geq 0$ on $M$. Then, $H(t)>0$ holds on $M$ for all $t \in (0,T]$.
\end{theorem}

\begin{proof}
We define $w \defr -H$ as well as 
\begin{align*}
\mathcal{L}w \defr -\partial^\square w + A \colon D_\Gamma^2 w + B \cdot \nabla_\Gamma w + Cw
\end{align*}
so that Assumptions \ref{ass_maxprinzip} are satisfied. We have 
\begin{align*}
\mathcal{L}w 
= -\mathcal{L}H
= \partial^\square H - \big( \Delta_\Gamma V + V\big| \nabla_\Gamma \nu \big|^2 \big)
= 0 \quad \text{ on } \, [0,T] \times M
\end{align*}
and $w(0) = -H(0) \leq 0$ on $M$. With Proposition \ref{MaxPrinzip_Folg},  $H(t) = -w(t) \geq 0$ follows on $M$ for all $t \in [0,T]$ and there exists $t_0 \in [0,T]$ with 
\begin{align*}
H(t) = -w(t) = 0 &\quad \text{ on } M \text{ for all } t \in (0,t_0] \text{ and } \\
H(t) = -w(t) > 0 &\quad \text{ on } M \text{ for all } t \in (t_0,T].
\end{align*}
As $\Gamma(t) = \theta_t(M)$ is a closed hypersurface, $H(t) = 0$ cannot hold on the whole surface $M$ (see Proposition \ref{mc_zero}). Hence, we have $t_0=0$ and then $H(t)>0$ follows on $M$ for all $t \in (0,T]$. 
\end{proof}

In the following remark, we state that the scaled mean curvature flow $V=g(c)H$ satisfies Assumptions \ref{ass_meanconvexity}. In particular, it thus conserves mean convexity and any initially mean convex hypersurface turns strictly mean convex instantly, i.e., $H(t,\cdot) > 0$ for any $t>0$. 

\begin{remark}\label{MeanConvex_Lemma_skalierterMCF} 
Let $\Gamma$ be a $(C^2$-$\phantom{.}C^2) \cap (C^1$-$\phantom{.}C^4)$-evolving immersed closed hypersurface with reference surface $M \subset \R^{d+1}$, normal $\nu$ and mean curvature $H$ that evolves with normal velocity
\begin{align*}
V = g(c) H,
\end{align*}
where $g \in C^2(\R)$ with $g>0$ and $c \in C^0\big([0,T],C^2(M)\big)$. (Notably, this is fulfilled for the usual mean curvature flow with $g \equiv 1$.) Then, Assumptions \ref{ass_meanconvexity} are satisfied. In particular, we have
\begin{align*}
&\Delta_\Gamma V + V \big| \nabla_\Gamma \nu \big|^2 
= A \colon D_\Gamma^2 H + B \cdot \nabla_\Gamma H + CH
\end{align*}
with 
\begin{align*}
A &\defr g(c) \Id, \phantom{xxx}
B \defr 2g'(c) \nabla_\Gamma c \phantom{xxx} \text{and} \\
C &\defr g(c)\big| \nabla_\Gamma \nu \big|^2 + g'(c)\Delta_\Gamma c + g''(c) |\nabla_\Gamma c|^2.
\end{align*}
The assumptions and $\nu \in C^0\big([0,T],C^1(M,\R^{d+1})\big)$ by \cite[Proposition 2.51]{Buerger} imply continuity of $A, B$ and $C$ on $[0,T] \times M$. Moreover, $A$ is clearly symmetric and as $g>0$ holds also positive definite on $[0,T] \times M$. 
\end{remark}

\subsection{Non-Conservation of Convexity}\label{Chap_Convex}

Huisken showed in \cite[Theorem 4.3]{Huisken} that, in addition to mean convexity, the usual mean curvature flow also conserves convexity. This result cannot be transferred to the scaled mean curvature flow: In contrast to the usual mean curvature flow where we have $\partial^\square H = \Delta_\Gamma H + \big| \nabla_\Gamma \nu\big|^2 H$, first order derivatives of $H$ occur in
\begin{align*}
\partial^\square H = g(c) \Delta_\Gamma H + 2 \nabla_\Gamma g(c) \cdot \nabla_\Gamma H + \big(\Delta_\Gamma g(c) + \big| \nabla_\Gamma \nu\big|^2 g(c) \big) H
\end{align*}
for the scaled mean curvature flow. In the evolution equation
\begin{align*}
\partial_t h_{ij} = g \Delta h_{ij} + \sum_{k,l} g^{kl} \nabla_i g \nabla_j h_{kl} + \text{ lower order terms }
\end{align*}
for the second fundamental form $[h_{ij}]$ (cf. \cite[Theorem 3.4]{Huisken} for the usual mean curvature flow), they produce additional first order terms in various directions. Therefore, Hamilton's maximum principle (\cite[Theorem 4.1]{Huisken}) cannot be applied to the second fundamental form and thus we cannot conclude the conservation of convexity as in \cite[Theorem 4.3]{Huisken}. \\
Actually, it is possible to construct examples of closed hypersurfaces evolving with the scaled mean curvature flow that loose their convexity in the course of time and we will do so in the following. Obviously, the dimension of the surface has to fulfill $d>1$, as otherwise convexity and mean convexity coincide and the latter is conserved by Section \ref{Chap_MeanConvex}. We choose $d=2$ and consider a rotationally symmetric structure for the hypersurface and the concentration defined thereon. However, this rotationally symmetric setting can be constructed in dimensions $d>2$ analogously. The idea of the construction is a hypersurface shaped as a long cylinder, whose first principal curvature on the sides is $0$ but the second is positive. Together, the surface thus has positive mean curvature and will shrink under the scaled mean curvature flow. A clever choice of the concentration forces the cylinder to shrink faster in the middle than at the ends, which turns the first principal curvature negative and therefore makes the surface non-convex. \\
We now construct this example explicitly. Fix an energy density $G \in C^{7}(\R)$ with $G''>0$ and $g \defr G - G' \cdot \Id > 0$. Let $1 \leq x_0 < x_1 < x_2 < x_3 < x_4 < x_5 \in \R$. We define the function $w_0: [x_0-1,x_5+1] \rightarrow \R_{\geq 0}$ with
\begin{align*}
w_0(x) \defr 
\begin{cases}
\big( 1-(x_0-x)^6 \big)^{1/6}, &\text{ if } x \in [x_0-1,x_0], \\
1, &\text{ if } x \in [x_0,x_5], \\
\big( 1-(x-x_5)^6 \big)^{1/6}, &\text{ if } x \in [x_5,x_5+1].
\end{cases}
\end{align*}
In particular, $w_0 \in C^5\big((x_0-1,x_5+1)\big)$ holds. Moreover, we define the surface of revolution
\begin{align*}
\Sigma \defr \Bigg\{ \gamma(x,\varphi) \defr 
\begin{bmatrix}   
x \\
w_0(x) \cos \varphi\\
w_0(x) \sin \varphi
\end{bmatrix}
\, \Bigg| \, x \in [x_0-1,x_5+1], \varphi \in [0,2\pi] \Bigg\}.
\end{align*}
This $\Sigma$ will be the reference surface as well as the initial surface. Furthermore, we choose $c_0 \in C^\infty\big([x_0-1,x_5+1]\big)$ with $c_{0,I} \leq c_0(x) \leq c_{0,O}$ for all $x \in [x_0-1,x_5+1]$ and
\begin{align*}
c_0(x) = 
\begin{cases}
c_{0,O}, &\text{ if } x \in [x_0-1,x_1], \\
c_{0,I}, &\text{ if } x \in [x_2,x_3], \\
c_{0,O}, &\text{ if } x \in [x_4,x_5+1]
\end{cases}
\end{align*}
for constant values $0<c_{0,I}<c_{0,O}$. As $c_0>0$ and thus $g'(c_0) = -G''(c_0)c_0 < 0$, also $g(c_0): [x_0-1,x_5+1] \rightarrow \R_{>0}$ holds with $g_I \defr g(c_{0,I}) \geq g(c_0)(x) \geq g(c_{0,O}) \defl g_O$ for all $x \in [x_0-1,x_5+1]$ and 
\begin{align*}
g(c_0)(x) = 
\begin{cases}
g_O, &\text{ if } x \in [x_0-1,x_1], \\
g_I, &\text{ if } x \in [x_2,x_3], \\
g_O, &\text{ if } x \in [x_4,x_5+1],
\end{cases}
\end{align*}
with $g_I > g_O > 0$. Finally, we define the rotationally symmetric function $u_0: \Sigma \rightarrow \R$ with 
\begin{align*}
u_0\big( \gamma(x,\varphi) \big) \defr c_0(x)
\end{align*}
for all $x \in [x_0-1,x_5+1]$ and $\varphi \in [0,2\pi]$. An illustration of the initial data $w_0$ and $g(c_0)$ can be found in Figure \ref{figure_w0,gc0}.

\begin{figure}[h!]
\begin{minipage}[c]{0.57\textwidth}
\begin{tikzpicture}[domain=-1:3]
\draw[->, dotted] (-0.7,0) -- (7.5,0);
\node[below] at (7.3,0) {$x$};
\foreach \x in {0,...,5}
  {\draw (\x+1,0.15) -- (\x+1,-0.15) node[below, fill=white, inner sep=2pt] {$x_{\x}$};}
\draw[->, dotted] (-0.3,-0.4) -- (-0.3,1.5);
\node[right] at (-0.3,1.4) {$w_0(x)$};
\draw[thick,domain=0:1,samples=500] plot(\x,{(abs(1-(1-\x)^3))^(1/3)});
\draw[thick,domain=1:6] plot(\x,1);
\draw[thick,domain=6:7,samples=500] plot(\x,{(abs(1-(abs(6-\x))^3))^(1/3)});
\draw[thick] (7,0) -- (7,0.2);
\end{tikzpicture} \\
$\phantom{.}$\\
\begin{tikzpicture}[domain=-1:3]
\draw[->, dotted] (-0.7,0) -- (7.5,0);
\node[below] at (7.3,0) {$x$};
\foreach \x in {0,...,5}
  {\draw (\x+1,0.15) -- (\x+1,-0.15) node[below, fill=white, inner sep=2pt] {$x_{\x}$};}
\draw[->, dotted] (-0.3,-0.4) -- (-0.3,1.5);
\node[right] at (-0.3,1.4) {$g\big(c_0(x)\big)$};
\draw[thick,domain=0:2] plot(\x,0.5);
\draw[thick,domain=2:3,samples=500] plot(\x,{1 + 5*(\x-3)^3 + 7.5*(\x-3)^4 + 3*(\x-3)^5});
\draw[thick,domain=3:4] plot(\x,1);
\draw[thick,domain=4:5,samples=500] plot(\x,{1 + 5*(4-\x)^3 + 7.5*(4-\x)^4 + 3*(4-\x)^5}); 
\draw[thick,domain=5:7] plot(\x,0.5); 
\end{tikzpicture}
\end{minipage}\hfill
\begin{minipage}{0.4\textwidth}
\vspace{\baselineskip}
\caption{Plot of the initial data. The function $w_0$ generates the initial surface $\Gamma_0$ via revolution, which is shaped like a long cylinder with smoothly closed endings. The initial concentration $c_0$ is chosen such that scaling the mean curvature flow with $g(c_0)$ increases the velocity of the flow in the middle of the cylinder compared to the sides.} \label{figure_w0,gc0}
\end{minipage}
\end{figure}
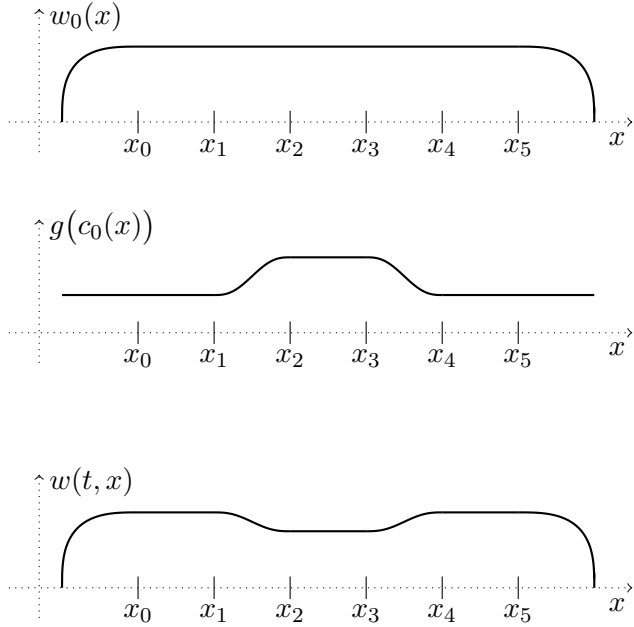

\begin{figure}[h!]
\begin{minipage}[c]{0.57\textwidth}
\begin{tikzpicture}[domain=-1:3]
\draw[->, dotted] (-0.7,0) -- (7.5,0);
\node[below] at (7.3,0) {$x$};
\foreach \x in {0,...,5}
  {\draw (\x+1,0.15) -- (\x+1,-0.15) node[below, fill=white, inner sep=2pt] {$x_{\x}$};}
\draw[->, dotted] (-0.3,-0.4) -- (-0.3,1.5);
\node[right] at (-0.3,1.4) {$w(t,x)$};
\draw[thick,domain=0:1,samples=500] plot(\x,{(abs(1-(1-\x)^3))^(1/3)});
\draw[thick,domain=1:2] plot(\x,1);
\draw[thick,domain=2:3,samples=500] plot(\x,{1 + 2.5*(2-\x)^3 + 3.75*(2-\x)^4 + 1.5*(2-\x)^5});
\draw[thick,domain=3:4] plot(\x,0.75);
\draw[thick,domain=4:5,samples=500] plot(\x,{1 + 2.5*(\x-5)^3 + 3.75*(\x-5)^4 + 1.5*(\x-5)^5});
\draw[thick,domain=5:6] plot(\x,1);
\draw[thick,domain=6:7,samples=500] plot(\x,{(abs(1-(abs(6-\x))^3))^(1/3)});
\draw[thick] (7,0) -- (7,0.2);
\end{tikzpicture}
\end{minipage}\hfill
\begin{minipage}{0.4\textwidth}
\vspace{\baselineskip}
\caption{A possible shape of the function $w$ at a time $t>0$. The generated surface of revolution obviously is not convex.} \label{figure_w}
\end{minipage}
\end{figure}

\begin{lemma}\label{meanconvex_sigma}
Let the surface $\Sigma$ and the functions $w_0$ and $u_0$ be as above. Then, $\Sigma$ is a $C^5$-embedded closed and convex hypersurface of $\R^3$ and we have $u_0 \in C^4(\Sigma)$.
\end{lemma}

\begin{proof}
Define the auxiliary function
\begin{align*}
\Phi(x,y,z) \defr (x-x_5)^6 + (y^2+z^2)^{6/2}
\end{align*}
for $[x,y,z] \in (x_5,\infty) \times \R^2$. Then, $\Phi \in C^\infty\big((x_5,\infty) \times \R^2 \big)$ holds and $1$ is a regular value of this function $\Phi$. Hence, 
\begin{align*}
\Phi^{-1}(1) 
&= \big\{ [x,y,z] \in (x_5,\infty) \times \R^2 \, \big| \, (y^2-z^2)^{6/2} = 1 - (x-x_5)^6 \big\} \\
&= \big\{ [x,y,z] \in (x_5,x_5+1] \times \R^2 \, \big| \, y^2+z^2 = \big( 1-(x-x_5)^6 \big)^{2/6} = w_0(x)^2 \big\} \\
&= \Bigg\{ 
\begin{bmatrix}   
x \\
w_0(x) \cos \varphi\\
w_0(x) \sin \varphi
\end{bmatrix}
\, \Bigg| \, x \in (x_5,x_5+1], \varphi \in [0,2\pi] \Bigg\}
\end{align*}
is a $2$-dimensional $C^\infty$-embedded submanifold of $\R^3$ (see e.g. \cite[Section 2.2, Proposition 2]{doCarmo}). In particular, the ``spherical shells'' at the ends of $\Sigma$ are $C^\infty$-embedded submanifolds. As $w_0$ and thus also $\Sigma$ are only of regularity $C^5$ in the ``glueing points'' $x_0$ and $x_5$, overall, $\Sigma$ is a $2$-dimensional $C^5$-embedded submanifold of $\R^3$. Clearly, $\Sigma$ also is convex, compact, connected and orientable, such that $\Sigma$ is a $C^5$-embedded closed and convex hypersurface of $\R^3$, and we have $u_0 \in C^4(\Sigma)$.
\end{proof}

Before we show that the evolution of the convex initial surface $\Gamma_0 \defr \Sigma$ turns non-convex over time, as an auxiliary step, we state formulas for the mean curvature and the normal velocity of surfaces of revolution, whose proofs can be found in \cite[Lemmas 5.5 and 5.6]{Buerger}.  

\begin{lemma}[Mean Curvature for Surfaces of Revolution]\label{NKonv_H} $\phantom{x}$ \\
Let $a,b \in \R$ with $a<b$. Furthermore, let $w \in C^2([a,b])$ with $w(a)=w(b)=0$ and $w>0$ on $(a,b)$ such that the surface of revolution 
\begin{align*}
\Gamma \defr \Bigg\{ \gamma(x,\varphi) \defr 
\begin{bmatrix}   
x \\
w(x) \cos \varphi\\
w(x) \sin \varphi
\end{bmatrix}
\, \Bigg| \, x \in [a,b], \varphi \in [0,2\pi] \Bigg\}
\end{align*}
is a $C^2$-embedded closed hypersurface in $\R^3$. Its mean curvature in every point $\gamma(x,\varphi)$ with $x \in (a,b)$ and $\varphi \in [0,2\pi]$ is given by
\begin{align*}
H_{|(x,\varphi)}
= \frac{1}{\sqrt{1+|w'(x)|^2}} \left( \frac{w''(x)}{1+|w'(x)|^2} - \frac{1}{w(x)} \right).
\end{align*}
\end{lemma}

\begin{lemma}[Normal Velocity for Surfaces of Revolution]\label{NKonv_V} $\phantom{x}$ \\
Let $T>0$ and $a,b \in C^0([0,T])$ with $a(t) < b(t)$ for all $t \in [0,T]$. Moreover, for every $t \in [0,T]$, let $w(t) \in C^2\big([a(t),b(t)]\big)$ with $\partial_t w(t) \in C^0\big( (a(t),b(t)) \big)$ if $t \in (0,T)$ and $w\big(t,a(t)\big) = w\big(t,b(t)\big)=0$ as well as $w(t)>0$ on $\big(a(t),b(t)\big)$ such that the surface of revolution
\begin{align*}
\Gamma \defr \Bigg\{ \{t\} \times \gamma(t,x,\varphi) \defr \{t\} \times
\begin{bmatrix}   
x \\
w(t,x) \cos \varphi\\
w(t,x) \sin \varphi
\end{bmatrix}
\, \Bigg| \, t \in [0,T], x \in \big[a(t),b(t)\big], \varphi \in [0,2\pi] \Bigg\}
\end{align*}
is a $C^1$-$\phantom{.}C^2$-evolving embedded closed hypersurface with $\Gamma(t) \subset \R^3$. Its normal velocity in every point $\big(t,\gamma(t,x,\varphi)\big)$ with $t \in (0,T)$, $x \in \big(a(t),b(t)\big)$ and $\varphi \in [0,2\pi]$ is given by
\begin{align*}
V_{|(t,x,\varphi)} 
= \frac{1}{\sqrt{1 + |\partial_x w(t,x)|^2}} \partial_t w(t,x). 
\end{align*}
\end{lemma}

\begin{theorem}[Non-Conservation of Convexity]\label{NKonv_Thm} $\phantom{x}$ \\
Let the energy density $G$, the surface $\Sigma$ and the function $u_0$ be as above. The initial hypersurface $\Gamma_0 \defr \Sigma$ is a convex surface whose evolution under \eqref{eq_Intro_GLS} with initial concentration $u_0$ does not stay convex. 
\end{theorem}

\begin{proof}
We know with Lemma \ref{meanconvex_sigma} that $\Gamma_0$ is convex. So, we only have to show that its evolution under \eqref{eq_Intro_GLS} with initial concentration $u_0$ turns non-convex. 
\begin{enumerate}
\item[Step 1:] \textit{Application of the short time existence result} \\
By assumption, we have $G \in C^7(\R)$ with $G''>0$ and $g>0$. Moreover, $\Sigma \subset \R^3$ is a $C^5$-embedded closed hypersurface and we have $\rho_0 \defr 0, u_0 \in C^4(\Sigma)$ with $\|\rho_0\|_{C^4(\Sigma)} < \delta_1$ and $\|\rho_0\|_{C^2(\Sigma)} < \delta_0$ for any $\delta_0, \delta_1 > 0$. For $T>0$ sufficiently small, Theorem \ref{lokEx_param} thus yields the existence of $\rho, u \in C^1\big([0,T],C^0(\Sigma)\big) \cap C^0\big([0,T],C^2(\Sigma)\big)$ with 
\begin{align*}
\left\{
\begin{aligned}
\partial_t \rho \phantom{bl} &= \phantom{bl} g(u)a(\rho)H(\rho) & &\text{ on } [0,T] \times \Sigma, \\
\partial_t u \phantom{bl} &= \phantom{bl} \Delta_{\Gamma_\rho} G'(u) + g(u) a(\rho) H(\rho) \nu_\Sigma \cdot \nabla_{\Gamma_\rho} u + g(u) H(\rho)^2 u & &\text{ on } [0,T] \times \Sigma, \\
\rho(0) \phantom{bl} &= \phantom{bl} \rho_0 & &\text{ on } \Sigma, \\
u(0) \phantom{bl} &= \phantom{bl} u_0 & &\text{ on } \Sigma
\end{aligned}
\right.
\end{align*}
as well as 
\begin{align}\label{Hoelder_Abschaetzung}
\|\rho(t)-\rho(0)\|_{C^2(M)} \leq RT^{1/4} 
\phantom{xx} \text{and} \phantom{xx}
\|u(t)-u(0)\|_{C^2(M)} \leq RT^{1/4}
\end{align}
for all $t \in [0,T]$ and some constant $R>0$.
With $\theta_\rho(t,z) \defr z + \rho(t,z)\nu_\Sigma(z)$ and $\Gamma_\rho(t) \defr \theta_\rho(t,\Sigma)$, hence the evolving hypersurface $\Gamma_\rho \defr \big\{ \{t\} \times \Gamma_\rho(t) \, \big| \, t \in [0,T] \big\}$ and the function $u \circ \theta_\rho^{-1}$ are a solution of \eqref{eq_Intro_GLS} with $\Gamma_\rho(0) = \Sigma = \Gamma_0$ and $u \circ \theta_\rho^{-1}(0) = u_0$. Due to the uniqueness property of the solution, the rotational symmetry of $\Sigma = \Gamma_0$ and $u_0$ implies that also $\Gamma_\rho(t)$ and $u(t)$ are rotationally symmetric. In particular, for every $t \in [0,T]$, there exists $w(t): [\tilde{x_0}(t),\tilde{x_5}(t)] \rightarrow \R$ with
\begin{align*}
\Gamma_\rho(t) = \Bigg\{ 
\begin{bmatrix}   
x \\
w(t,x) \cos \varphi\\
w(t,x) \sin \varphi
\end{bmatrix}
\, \Bigg| \, x \in [\tilde{x_0}(t),\tilde{x_5}(t)], \varphi \in [0,2\pi] \Bigg\}
\end{align*}
and a function $c: [0,T] \times [x_0-1,x_5+1] \rightarrow \R$ with $u\big(t,\gamma(x,\varphi)\big) = c(t,x)$ for every $x \in [x_0-1,x_5+1]$ and $\varphi \in [0,2\pi]$. For $T>0$ sufficiently small, we can assume $[x_0,x_5] \subset [\tilde{x_0}(t),\tilde{x_5}(t)]$ for all $t \in [0,T]$ as well as $w > 0$ on $[0,T] \times [x_0,x_5]$. Then we have $c \in \E_{1,T} \defr C^1\big([0,T],C^0([x_0,x_5])\big) \cap C^0\big([0,T],C^2([x_0,x_5])\big)$. On account of $\nu_\Sigma \circ \gamma(x,\varphi) = (0,\cos \varphi, \sin \varphi)$ on $[x_0,x_5]$, 
\begin{align*}
w(t,x) = 1 + \rho\big(t,\gamma(x,\varphi)\big)
\end{align*}
follows for all $x \in [x_0,x_5]$ and $\varphi \in [0,2\pi]$ and hence we also have $w \in \E_{1,T}$. 
\item[Step 2:] \textit{Estimation of the Hölder-functions} \\
Due to Estimates \eqref{Hoelder_Abschaetzung}, we have 
\begin{align*}
\|w(t)-w_0\|_{C^0([x_0,x_5])} + \|\partial_x w(t)\|_{C^0([x_0,x_5])} + \|\partial_{xx} w(t)\|_{C^0([x_0,x_5])} &\leq R T^{1/4} \\
\text{and} \phantom{xx} \|c(t)-c_0\|_{C^0([x_0,x_5])} &\leq R T^{1/4}
\end{align*}
for all $t \in [0,T]$. For $T>0$ sufficiently small, we can assume $0 \leq c \leq 2c_{0,O}$ and then $g \in C^1(\R)$ implies 
\begin{align*}
\big\|g\big(c(t)\big)-g(c_0)\|_{C^0([x_0,x_5])} 
\leq \|g\|_{C^1([0,2c_{0,O}])} \|c(t)-c_0\|_{C^0([x_0,x_5])}
\lesssim R T^{1/4}
\end{align*}
for all $t \in [0,T]$. Hence, for arbitrary $\tilde{\varepsilon}>0$, we can choose $T>0$ sufficiently small such that for any $(t,x) \in [0,T] \times [x_0,x_5]$, 
\begin{align*}
\frac{\partial_{xx} w(t,x)}{1 + |\partial_x w(t,x)|^2} - \frac{1}{w(t,x)}
\begin{cases}
\leq \frac{R T^{1/4}}{1+0} - \frac{1}{w_0(x)+R T^{1/4}} &\leq -\frac{1}{w_0(x)} + \tilde{\varepsilon}, \\
\geq \frac{-R T^{1/4}}{1+R^2 T^{1/2}} - \frac{1}{w_0(x) - R T^{1/4}} &\geq - \frac{1}{w_0(x)} - \tilde{\varepsilon}
\end{cases}
\end{align*}
as well as
\begin{align*}
g\big(c(t,x)\big) 
\begin{cases}
\leq g\big(c_0(x)\big) + \tilde{\varepsilon}, \\
\geq g\big(c_0(x)\big) - \tilde{\varepsilon}
\end{cases}
\end{align*}
hold. Overall,
\begin{align*}
g\big(c(t,x)\big) \left( \frac{\partial_{xx} w(t,x)}{1 + |\partial_x w(t,x)|^2} - \frac{1}{w(t,x)} \right)
\begin{cases}
\leq -g\big(c_0(x)\big) \frac{1}{w_0(x)} + \varepsilon, \\
\geq -g\big(c_0(x)\big) \frac{1}{w_0(x)} - \varepsilon
\end{cases}
\end{align*}
follows, where we can choose $\varepsilon>0$ sufficiently small such that $g_O + 2\varepsilon < g_I$ is valid. \\
Lemmas \ref{NKonv_H} and \ref{NKonv_V} yield that rotationally symmetric solutions of \eqref{eq_Intro_GLS} fulfill
\begin{align*}
\partial_t w = g(c) \left( \frac{\partial_{xx} w}{1 + |\partial_x w|^2} - \frac{1}{w} \right).
\end{align*}
So, for every $(t,x) \in [0,T] \times [x_0,x_5]$,
\begin{align*}
w(t,x) 
= w_0(x) + \int_0^t \partial_t w(s,x) \mathrm{d}s
\begin{cases}
\leq w_0(x) + \left( -g\big(c_0(x)\big) \frac{1}{w_0(x)} + \varepsilon \right) \cdot t, \\
\geq w_0(x) + \left( -g\big(c_0(x)\big) \frac{1}{w_0(x)} - \varepsilon \right) \cdot t
\end{cases}
\end{align*}
holds. For all $x \in [x_2,x_3]$ and all $y \in [x_0,x_1] \cup [x_4,x_5]$, we have $w_0(x)=w_0(y)=1$ and $g\big(c_0(x)\big)=g_I$, $g\big(c_0(y)\big)=g_O$; and hence for every $t \in (0,T]$
\begin{align*}
w(t,y) 
&\geq w_0(y) + \left( -g\big(c_0(y)\big) \frac{1}{w_0(y)} - \varepsilon \right) \cdot t 
= 1 + ( -g_O - \varepsilon ) \cdot t \\
&> 1 + ( -g_I + \varepsilon ) \cdot t
= w_0(x) + \left( -g\big(c_0(x)\big) \frac{1}{w_0(x)} + \varepsilon \right) \cdot t
\geq w(t,x)
\end{align*}
follows. Thus, $\Gamma_\rho(t)$ is not convex for $t \in (0,T]$. A possible shape of the function $w(t,\cdot)$ is illustrated in Figure \ref{figure_w}. \qedhere
\end{enumerate}
\end{proof}

\subsection{Formation of Self-Intersections}\label{ChapSI}

The usual mean curvature flow does not allow for self-intersections: As a consequence of the maximum principle for parabolic differential equations, an initially embedded surface will remain embedded as long as it exists. This property does not transfer to our scaled mean curvature flow and we develop a concrete example for the occuring of self-intersections in this section. \\
An embedded hypersurface can obviously never have a self-intersection. As an initially embedded surface will stay embedded at least for a small time, a short-time existence result only for the case of embedded surfaces would provide the evolution of this surface only as long as there does not occur a self-intersection (yet). To describe self-intersections it is thus necessary to use the theory of immersed hypersurfaces and it is crucial that we proved the short-time existence result also for the case of immersed surfaces. \\ 
The idea is to start with a very thin, curved tube as in Figure \ref{figure_Gamma_0}. The curvature forces both sides of the tube to move to the right. Choosing an initial concentration that increases the evolution for the left side in comparison with the right side of the tube will produce a self-intersection. To rigorously prove this, we have to ensure that firstly the evolution provided by the short-time existence result lasts long enough such that the self-intersection will occur during that time and secondly that the concentration does not distribute too fast but that the difference in concentration from the left to the right side of the tube persists long enough such that the self-intersection will form. 
Both of this will be achieved by choosing the initial tube sufficiently thin. Therefore, we need to ensure that both the final time (up to which the existence result guarantees the existence of an evolution) as well as the evolution driving force (i.e. the concentration) can be controlled independently of the initial thickness of the tube. 
For this reason, we formulated the short-time existence result in a version that allows for small changes in the initial hypersurface. In particular, it yields a final time and a bound on the solution which are both independent of the initial thickness of the tube. The latter can then be used to control the driving force. \\
Now, we construct this example concretely. We set $d=1$, so we will be dealing with curves instead of hypersurfaces, but the example can easily be extended to more dimensions using rotation arguments. 
First, we fix an energy density function $G \in C^7(\R)$ with $G''>0$ and $g \defr G-G' \cdot \Id >0$. Additionally, we assume that $g$ is not a constant function, because otherwise we would obtain a constantly scaled mean curvature flow which of course never develops self-intersections. 
As reference surface, we choose $\Sigma=F(M)$ as illustrated in Figure \ref{figure_Sigma}: $\Sigma$ is an immersed curve, consisting of a circular arc with radius $R>0$ that is smoothly connected at the endings such that it forms a closed curve. A possibility of splitting the arc smoothly is to add appropriate exponential terms to the parameterization of the arc, as illustrated in Figure \ref{figure_split}. Then, the free ends can be connected easily to form a curve of arbitrary smoothness. We choose $\Sigma$ to be of differentiability $C^5$ at least. \\
By construction, there exist two preimages of $(0,0) \in \Sigma$ in $M$, which we call $z_l$ and $z_r$. Choosing the sign of the unit normal field $\nu_\Sigma$ of $\Sigma$ as in Figure \ref{figure_Sigma}, we get
\begin{align*}
\begin{alignedat}{3}
F(z_l) &= (0,0), \phantom{xxx}& F(z_r) &= (0,0), \\
\nu_\Sigma(z_l) &= (-1,0), \phantom{xxx}& \nu_\Sigma(z_r) &= (+1,0).
\end{alignedat}
\end{align*}
We fix a constant height function $\rho_0 > 0$ which will be scaled with a small $\varepsilon \in (0,1)$ and we define $\rho_0^\varepsilon \defr \varepsilon \rho_0$. Finally, we choose an initial concentration $u_0 \in C^4(M)$ with 
\begin{align*}
g\big(u_0(z_l)\big) > g\big(u_0(z_r)\big),
\end{align*}
which is possible as $g$ is not constant.

\begin{figure}[!htb]
\begin{minipage}[t]{0.45\textwidth}
\begin{center}
\begin{tikzpicture}[domain=-1:2.3]
\draw[->, dotted] (-0.8,0) -- (3,0);
\draw[->, dotted] (0,-2.3) -- (0,2.3);
\draw[thick] ([shift=(90:2)]2,0) arc (90:270:2);
\draw[thick] ([shift=(-90:0.25)]2.25,2) arc(-90:90:0.25);
\draw[thick, domain=2:2.25] plot(\x,{306-432*\x+204*\x^2-32*\x^3});
\draw[thick, domain=2:2.25] plot(\x,{-302+432*\x-204*\x^2+32*\x^3});
\draw[thick] ([shift=(-90:0.25)]2.25,-2) arc(-90:90:0.25);
\draw[thick, domain=2:2.25] plot(\x,{302-432*\x+204*\x^2-32*\x^3});
\draw[thick, domain=2:2.25] plot(\x,{-306+432*\x-204*\x^2+32*\x^3});
\node[right] at (2.5,2) {$\Sigma$};
\draw[blue,->] (2,0) -- (1.5,1.94) node[below] {$\phantom{xxx} R$};
\draw[red,->] (0,0) -- (0.5,0) node[below] {$\phantom{xx} \nu_\Sigma(z_r)$};
\draw[red,->] (0,0) -- (-0.5,0) node[below] {$\nu_\Sigma(z_l) \phantom{x}$};
\draw[mygreen, fill=mygreen]  (0,0) circle (2pt) node[above] {$(0,0) \phantom{xxxx}$};
\node[white,above] at (1,1.95) {$R+\varepsilon\rho_0$};
\draw[white,thick, domain=2:2.25, samples=10] plot(\x,{-306.1+432*\x-204*\x^2+32*\x^3});
\end{tikzpicture}
\end{center}
  \caption{The immersed reference curve $\Sigma$, consisting of a circular arc with radius $R$ that is smoothly connected at the endings. The points $z_l,z_r \in M$ are both preimages of $(0,0) \in \Sigma$. We choose the sign of the unit normal $\nu_\Sigma$ such that it points outwards in a neighborhood of $z_l$ and inwards in a neighborhood of $z_r$.}\label{figure_Sigma}
\end{minipage} \hfill
\begin{minipage}[t]{0.45\textwidth}
\begin{center}
\begin{tikzpicture}[domain=-1:3]
\draw[->, dotted] (-0.8,0) -- (3,0);
\draw[->, dotted] (0,-2.3) -- (0,2.3);
\draw[thick] ([shift=(90:2.1)]2,0) arc (90:270:2.1);
\draw[thick] ([shift=(90:1.9)]2,0) arc (90:270:1.9);
\draw[thick] ([shift=(-90:0.35)]2.25,2) arc(-90:90:0.35);
\draw[thick, domain=2:2.25] plot(\x,{306.1-432*\x+204*\x^2-32*\x^3});
\draw[thick, domain=2:2.25] plot(\x,{-302.1+432*\x-204*\x^2+32*\x^3});
\draw[thick] ([shift=(-90:0.35)]2.25,-2) arc(-90:90:0.35);
\draw[thick, domain=2:2.25] plot(\x,{302.1-432*\x+204*\x^2-32*\x^3});
\draw[thick, domain=2:2.25] plot(\x,{-306.1+432*\x-204*\x^2+32*\x^3});
\node[right] at (2.6,2) {$\Gamma_0^\varepsilon$};
\draw[<->] (0.82,1.49) -- (0.7,1.65);
\node[left] at (0.9,1.85) {$2\varepsilon\rho_0$};
\draw[blue,->] (2,0) -- (1.3,1.99);
\node[blue,above] at (1,1.95) {$R+\varepsilon\rho_0$};
\draw[blue,->] (2,0) -- (1.7,1.88);
\node[blue,right] at (1.7,1.4) {$R-\varepsilon\rho_0$};
\draw[red,->] (-0.1,0) -- (-0.6,0) node[below] {$\nu_{\rho^\varepsilon}(0,z_l) \phantom{xxx}$};
\draw[red,->] (0.1,0) -- (0.6,0) node[below] {$\phantom{xxxx} \nu_{\rho^\varepsilon}(0,z_r)$};
\draw[mygreen, fill=mygreen]  (-0.1,0) circle (2pt) node[above] {$\theta_{\rho^\varepsilon}(0,z_l) \phantom{xxxxxxx}$};
\draw[mygreen, fill=mygreen]  (0.1,0) circle (2pt) node[above] {$\phantom{xxxxxxxx} \theta_{\rho^\varepsilon}(0,z_r)$};
\end{tikzpicture}
\end{center}
  \caption{The initial curve $\Gamma_0^\varepsilon$, consisting of two smoothly connected circular arcs with radii $R \pm \varepsilon \rho_0$. Its unit normal $\nu_{\rho^\varepsilon} \circ \big(\theta_{\rho^\varepsilon}(0,\cdot)\big)^{-1}$ points outwards on the left arc and inwards on the right arc. As $\Gamma_0^\varepsilon$ is embedded, the images of $z_l, z_r \in M$ on $\Gamma_0^\varepsilon$ do not coincide.}\label{figure_Gamma_0}
\end{minipage} 
\end{figure}

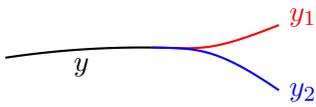
\begin{figure}[!htb]
\begin{minipage}[c]{0.3\textwidth}
\begin{tikzpicture}[domain=-1:3, scale=4]
\draw[thick,red,domain=1.01:1.15,samples=500] plot(3*\x,{sqrt(9-(3-3*\x)^2)+exp(1/(3-3*\x))});
\draw[thick,blue,domain=1.01:1.15,samples=500] plot(3*\x,{sqrt(9-(3-3*\x)^2)-exp(1/(3-3*\x))});
\draw[thick,domain=0.85:1.01,samples=500] plot(3*\x,{sqrt(9-(3-3*\x)^2)});
\node[above] at (2.8,2.87) {$y$};
\node[red,right] at (3.45,3.1) {$y_1$};
\node[blue,right] at (3.45,2.85) {$y_2$};
\end{tikzpicture}
\end{minipage}\hfill
\begin{minipage}[c]{0.65\textwidth}
\vspace{\baselineskip}
\caption{After reparameterization, $y(x) = \sqrt{R^2-x^2}$ parameterizes the circular arc. Adding appropriate exponential terms leads to smooth continuations $y_1(x) = \sqrt{R^2-x^2} + \exp \left(\frac{-1}{x} \right)$ and $y_2(x) = \sqrt{R^2-x^2} - \exp \left(\frac{-1}{x} \right)$, $x>0$, of $y$ that split the arc into two parts.} \label{figure_split}
\end{minipage}
\end{figure}

\begin{figure}[!htb]
$\phantom{.}$
\begin{center}
\begin{minipage}[c]{0.7\textwidth}
\begin{center}
\begin{minipage}[c]{0.23\textwidth}
\begin{tikzpicture}[domain=0:2.25]
\draw[thick,domain=2:2.25] plot(\x,{-306+432*\x-204*\x^2+32*\x^3});
\draw[thick,domain=2:2.25] plot(\x,{306-432*\x+204*\x^2-32*\x^3});
\draw[thick] ([shift=(90:2)]2,0) arc (90:270:2);
\draw[thick] ([shift=(-90:0.25)]2.25,2) arc(-90:90:0.25);
\draw[thick, domain=2.15:2.25] plot(\x,{306-432*\x+204*\x^2-32*\x^3});
\draw[thick, domain=2:2.25] plot(\x,{-302+432*\x-204*\x^2+32*\x^3});
\draw[thick] ([shift=(-90:0.25)]2.25,-2) arc(-90:90:0.25);
\draw[thick, domain=2:2.25] plot(\x,{302-432*\x+204*\x^2-32*\x^3});
\draw[thick, domain=2.15:2.25] plot(\x,{-306+432*\x-204*\x^2+32*\x^3});
\node[right] at (0,0) {$\Sigma$};
\end{tikzpicture}
\end{minipage} 
\begin{minipage}[c]{0.05\textwidth}
\hfill
\end{minipage}
\begin{minipage}[c]{0.23\textwidth}
\begin{tikzpicture}[domain=0:2.25]
\draw[red,thick] ([shift=(90:2)]2,0) arc (90:270:2);
\draw[red,thick] ([shift=(-90:0.25)]2.25,2) arc(-90:90:0.25);
\draw[red,thick, domain=2:2.25] plot(\x,{306-432*\x+204*\x^2-32*\x^3});
\draw[red,thick, domain=2.15:2.25] plot(\x,{-302+432*\x-204*\x^2+32*\x^3});
\draw[red,thick] ([shift=(-90:0.25)]2.25,-2) arc(-90:90:0.25);
\draw[red,thick, domain=2.15:2.25] plot(\x,{302-432*\x+204*\x^2-32*\x^3});
\draw[red,thick, domain=2:2.25] plot(\x,{-306+432*\x-204*\x^2+32*\x^3});
\node[red,right] at (0,0) {$F(U_1)$};
\end{tikzpicture}
\end{minipage} 
\begin{minipage}[c]{0.05\textwidth}
\hfill
\end{minipage}
\begin{minipage}[c]{0.23\textwidth}
\begin{tikzpicture}[domain=0:2.25]
\draw[mygreen,thick] ([shift=(90:2)]2,0) arc (90:270:2);
\draw[mygreen,thick] ([shift=(-90:0.25)]2.25,2) arc(-90:90:0.25);
\draw[mygreen,thick, domain=2.15:2.25] plot(\x,{306-432*\x+204*\x^2-32*\x^3});
\draw[mygreen,thick, domain=2:2.25] plot(\x,{-302+432*\x-204*\x^2+32*\x^3});
\draw[mygreen,thick] ([shift=(-90:0.25)]2.25,-2) arc(-90:90:0.25);
\draw[mygreen,thick, domain=2:2.25] plot(\x,{302-432*\x+204*\x^2-32*\x^3});
\draw[mygreen,thick, domain=2.15:2.25] plot(\x,{-306+432*\x-204*\x^2+32*\x^3});
\node[mygreen,right] at (0,0) {$F(U_2)$};
\end{tikzpicture}
\end{minipage} 
\begin{minipage}{0.025\textwidth}
\hfill
\end{minipage}
\end{center}
\begin{center} 
\caption{The immersed surface $\Sigma$ can be covered by two embedded patches, called $F(U_1)$ and $F(U_2)$ with $U_1,U_2 \subset M$.} \label{figure_embeddedpatches}
\end{center}
\end{minipage}
\end{center}
\end{figure}
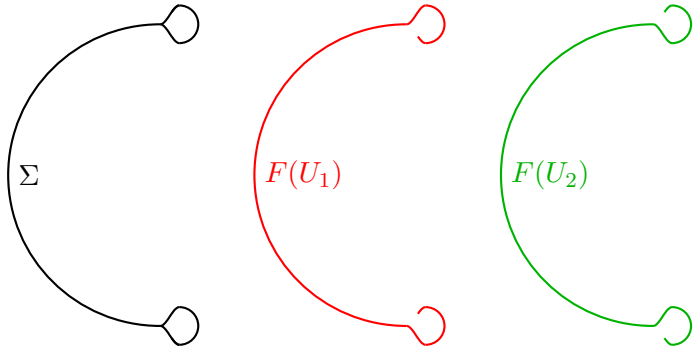

\begin{theorem}[Formation of Self-Intersections] $\phantom{x}$ \\
Let the energy density $G$, the reference curve $\Sigma = F(M)$ and the functions $\rho_0$ and $u_0$ be as above. For sufficiently small $\varepsilon>0$, the initial curve 
\begin{align*}
\Gamma_0^\varepsilon = \big\{ F(z) + \varepsilon\rho_0\nu_\Sigma(z) \, \big| \, z \in M \big\}
\end{align*}
is an embedded curve whose evolution under \eqref{eq_Intro_GLS} with initial concentration $u_0$ leads to a self-intersection.
\end{theorem}

\begin{proof}
The short-time existence result (Theorem \ref{lokEx_param}) yields the existence of a $T>0$ and an $\varepsilon_0>0$ such that for all $\varepsilon \in (0,\varepsilon_0]$ there exists a solution $(\rho^\varepsilon, u^\varepsilon): [0,T] \times M \rightarrow \R^2$ of 
\begin{align*}
\left\{
\begin{aligned}
\partial_t \rho^\varepsilon \phantom{bl} &= \phantom{bl} g(u^\varepsilon)a(\rho^\varepsilon)H(\rho^\varepsilon) & &\text{ on } [0,T] \times M, \\
\partial_t u^\varepsilon \phantom{bl} &= \phantom{bl} \Delta_{\Gamma_{\rho^\varepsilon}} G'(u^\varepsilon) + g(u^\varepsilon) a(\rho^\varepsilon) H(\rho^\varepsilon) \nu_\Sigma \cdot \nabla_{\Gamma_{\rho^\varepsilon}} u^\varepsilon + g(u^\varepsilon) H(\rho^\varepsilon)^2 u^\varepsilon & &\text{ on } [0,T] \times M, \\
\rho^\varepsilon(0) \phantom{bl} &= \phantom{bl} \rho_0^\varepsilon = \varepsilon \rho_0 & &\text{ on } M, \\
u^\varepsilon(0) \phantom{bl} &= \phantom{bl} u_0 & &\text{ on } M
\end{aligned}
\right.
\end{align*}
with $\rho^\varepsilon, u^\varepsilon \in \E_{1,T} = C^1\big([0,T],C^0(M)\big) \cap C^0\big([0,T],C^2(M)\big)$. Independently of $\varepsilon \in (0,\varepsilon_0]$,
\begin{align}\label{Hoelder_Abschaetzung2}
\|\partial_t \rho^\varepsilon(t)-\partial_t \rho^\varepsilon(0)\|_{C^0(M)} \leq R^hT^{1/4}
\end{align}
holds for all $t \in [0,T]$ and some constant $R^h>0$. \\
The evolving curve is described by the global parameterization 
\begin{align*}
\theta_{\rho^\varepsilon}: [0,T] \times M \rightarrow \R^2, \phantom{xx} \theta_{\rho^\varepsilon}(t,z) \defr F(z) + \rho^\varepsilon(t,z)\nu_\Sigma(z).
\end{align*}
\begin{enumerate}
\item[Step 1:] \textit{Embeddedness of the initial hypersurface} \\
As $\Sigma$ is closed, it can be covered by finitely many embedded patches. As shown in Figure \ref{figure_embeddedpatches}, two embedded patches are sufficient to cover $\Sigma$. Let $U_1,U_2 \subset M$ be the preimages of these embedded patches $F(U_1), F(U_2)$. We can choose $\varepsilon>0$ sufficiently small, such that for both embedded patches $F(U_i) \subset \Sigma$, 
\begin{align*}
\theta_{\rho^\varepsilon}(0,\cdot) \circ F^{-1}: F(U_i) \rightarrow \R^2, 
\phantom{xx} p \mapsto p + \rho_0^\varepsilon \nu_\Sigma\big(F^{-1}(p)\big) = p + \rho_0^\varepsilon \nu_{F(U_i)}(p)
\end{align*}
is an embedding (see \cite[Proposition 2.60]{Buerger}, where we assume w.l.o.g. that $F(U_i)$ is expanded to an embedded closed hypersurface.). In particular, for any $z_1, z_2 \in U_i$ with $z_1 \neq z_2$ we have $F(z_1) \neq F(z_2)$ and thus $\theta_{\rho^\varepsilon}(0,z_1) \neq \theta_{\rho^\varepsilon}(0,z_2)$. For any $z_1 \in U_1 \setminus U_2$ and $z_2 \in U_2 \setminus U_1$, we clearly have 
\begin{align*}
\theta_{\rho^\varepsilon}(0,z_1) = F(z_1) + \rho_0^\varepsilon \nu_\Sigma(z_1) \neq F(z_2) + \rho_0^\varepsilon \nu_\Sigma(z_2) = \theta_{\rho^\varepsilon}(0,z_2)
\end{align*}
as the initial height function $\rho_0^\varepsilon = \varepsilon \rho_0$ is positive everywhere and $\nu_{\Sigma \, |U_1 \setminus U_2}$ points outwards whereas $\nu_{\Sigma \, |U_2 \setminus U_1}$ points inwards, and so $z_1$ and $z_2$ are driven apart by $\theta_{\rho^\varepsilon}(0,\cdot)$. Altogether, if $\varepsilon>0$ is sufficiently small, for any $z_1, z_2 \in M$ with $z_1 \neq z_2$ also $\theta_{\rho^\varepsilon}(0,z_1) \neq \theta_{\rho^\varepsilon}(0,z_2)$ holds. 
This implies that for $\varepsilon>0$ sufficiently small, 
\begin{align*}
\theta_{\rho^\varepsilon}(0, \cdot): M \rightarrow \R^2, \phantom{xx} \theta_{\rho^\varepsilon}(0,z) \defr F(z) + \rho_0^\varepsilon\nu_\Sigma(z)
\end{align*}
is injective and thus an embedding. Therefore, the initial curve 
\begin{align*}
\Gamma_0^\varepsilon = \big\{ \theta_{\rho^\varepsilon}(0,z) \, \big| \, z \in M \big\}
\end{align*}
is an embedded curve. The curve $\Gamma_0^\varepsilon$ is illustrated in Figure \ref{figure_Gamma_0} and we collect some of its geometric quantities now. Because $\rho_0^\varepsilon$ is constant, the curve $\Gamma_0^\varepsilon$ consists of two circular arcs, the inner one with radius $R-\varepsilon\rho_0$ and the outer one with radius $R+\varepsilon\rho_0$. Due to the chosen sign of the normal, the (mean) curvature of the inner arc is positive and that of the outer arc is negative. Especially in our fixed points $z_l$ and $z_r$, we thus have
\begin{align*}
H(\rho^\varepsilon)(0,z_l) = \frac{-1}{R+\varepsilon \rho_0} 
\phantom{xxx}\text{ and }\phantom{xxx}
H(\rho^\varepsilon)(0,z_r) = \frac{1}{R-\varepsilon \rho_0}
\end{align*}
as well as $\nu_{\rho^\varepsilon}(0,z_l) = (-1,0) = \nu_\Sigma(z_l)$ and $\nu_{\rho^\varepsilon}(0,z_r) = (+1,0) = \nu_\Sigma(z_r)$. So, 
\begin{align*}
a(\rho^\varepsilon)(0,z_l) = 1 
\phantom{xxx}\text{ and }\phantom{xxx}
a(\rho^\varepsilon)(0,z_r) = 1
\end{align*}
follow. 
\item[Step 2:] \textit{Formation of self-intersection} \\
We want to show that the evolution of $\Gamma_0^\varepsilon$ leads to a self-intersection. For any $t \in [0,T]$, 
\begin{align*}
\theta_{\rho^\varepsilon}(t,z_l) &= F(z_l) + \rho^\varepsilon(t,z_l)\nu_\Sigma(z_l) = (0,0) + \rho^\varepsilon(t,z_l)(-1,0) = \big( -\rho^\varepsilon(t,z_l), 0 \big), \\
\theta_{\rho^\varepsilon}(t,z_r) &= F(z_r) + \rho^\varepsilon(t,z_r)\nu_\Sigma(z_r) = (0,0) + \rho^\varepsilon(t,z_r)(+1,0) = \big( \rho^\varepsilon(t,z_r), 0 \big)
\end{align*}
holds. In particular,
\begin{align*}
\big[ \theta_{\rho^\varepsilon}(0,z_l) \big]_1 = - \varepsilon \rho_0 < \varepsilon \rho_0 = \big[ \theta_{\rho^\varepsilon}(0,z_r) \big]_1
\end{align*}
holds and if we have 
\begin{align}\label{Eq_Intersection}
\big[ \theta_{\rho^\varepsilon}(T_0,z_l) \big]_1 > \big[ \theta_{\rho^\varepsilon}(T_0,z_r) \big]_1
\Leftrightarrow
-\rho^\varepsilon(T_0,z_l) > \rho^\varepsilon(T_0,z_r)
\end{align}
for a $T_0 \in (0,T]$, then a self-intersection with $\theta_{\rho^\varepsilon}(T_1,z_l) = \theta_{\rho^\varepsilon}(T_1,z_r)$ occurred at a time $T_1 \in (0,T_0)$. All that is left to prove is thus the existence of a $T_0 \in (0,T]$ with \eqref{Eq_Intersection} for sufficiently small $\varepsilon>0$. \\ 
As we have $g\big(u_0(z_l)\big) > g\big(u_0(z_r)\big)$, there exists a $K \in \R_{>0}$ with 
\begin{align}\label{eq_g}
g\big(u_0(z_l)\big) = g\big(u_0(z_r)\big) + K.
\end{align}
Choose $T_0 \in (0,T]$ so small that $R^hT_0^{1/4} \leq \frac{K}{8R}$. Then, with Estimate \eqref{Hoelder_Abschaetzung2},
\begin{align}\label{eq_dtrho}
\big\| \partial_t \rho^\varepsilon(t,\cdot) - \partial_t \rho^\varepsilon(0,\cdot) \big\|_{C^0(M)}
\leq R^hT_0^{1/4}
\leq \frac{K}{8R}
\end{align}
holds for all $t \in [0,T_0]$ and $\varepsilon \in (0,\varepsilon_0]$. 
Now, choose $\varepsilon = \varepsilon(T_0) \in (0,\varepsilon_0]$ so small that $\varepsilon \leq \frac{R}{\sqrt{2}\rho_0}$ and that 
\begin{align*}
\varepsilon < \frac{T_0KR}{8\rho_0\big( R^2 + 2g(u_0(z_r)) \big)}.
\end{align*}
Then, with $H(\rho^\varepsilon_0)(z_l) = \frac{-1}{R+\varepsilon\rho_0}$ and $H(\rho^\varepsilon_0)(z_r) = \frac{1}{R-\varepsilon\rho_0}$,
\begin{align}\label{eq_H}
-H(\rho^\varepsilon_0)(z_l) \geq \frac{1}{2R} 
\phantom{xx} \text{ and } \phantom{xx} 
H(\rho^\varepsilon_0)(z_l) + H(\rho^\varepsilon_0)(z_r)
= \frac{2\varepsilon\rho_0}{R^2-\varepsilon^2\rho_0^2}
\leq \frac{4\varepsilon\rho_0}{R^2}
\end{align}
hold and we have 
\begin{align}\label{eq_epsT0}
2\varepsilon\rho_0 \left( 1 + \frac{2}{R^2} g\big(u_0(z_r)\big) \right)
< \frac{T_0K}{4R}.
\end{align}
With these preliminary considerations and $a(\rho^\varepsilon_0)(z_l) = a(\rho^\varepsilon_0)(z_r)=1$, we can compute
\begin{align*}
-\rho^\varepsilon(T_0,z_l)
&= - \left( \rho^\varepsilon_0 + \int_0^{T_0} \partial_t \rho^\varepsilon(t,z_l) \, \mathrm{d}t \right) \\
&\stackrel{\eqref{eq_dtrho}}{\geq} -\rho^\varepsilon_0 -\frac{T_0K}{8R} -T_0\partial_t \rho^\varepsilon(0,z_l) \\
&= -\varepsilon\rho_0 -\frac{T_0K}{8R} +T_0g\big(u_0(z_l)\big)a(\rho^\varepsilon_0)(z_l)\big(-H(\rho^\varepsilon_0)(z_l)\big) \\
&\stackrel{\eqref{eq_g}}{=} -\varepsilon\rho_0 -\frac{T_0K}{8R} +T_0K\big(-H(\rho^\varepsilon_0)(z_l)\big) +T_0g\big(u_0(z_r)\big)\big(-H(\rho^\varepsilon_0)(z_l)\big) \\
&\stackrel{\eqref{eq_H}}{\geq} -\varepsilon\rho_0 -\frac{T_0K}{8R} +\frac{T_0K}{2R} -\frac{4\varepsilon\rho_0T_0}{R^2}g\big(u_0(z_r)\big) +T_0g\big(u_0(z_r)\big)H(\rho^\varepsilon_0)(z_r) \\
&\geq -\varepsilon\rho_0 +\frac{3T_0K}{8R} -\frac{4\varepsilon\rho_0}{R^2}g\big(u_0(z_r)\big) +T_0\partial_t\rho^\varepsilon(0,z_r) \\
&\stackrel{\eqref{eq_dtrho}}{\geq} -\varepsilon\rho_0 +\frac{3T_0K}{8R} -\frac{4\varepsilon\rho_0}{R^2}g\big(u_0(z_r)\big) -\frac{T_0K}{8R} + \int_0^{T_0} \partial_t\rho^\varepsilon(t,z_r) \mathrm{d}t \\
&= -2\varepsilon\rho_0 +\frac{T_0K}{4R} -\frac{4\varepsilon\rho_0}{R^2}g\big(u_0(z_r)\big) + \left( \rho^\varepsilon_0 + \int_0^{T_0} \partial_t\rho^\varepsilon(t,z_r) \mathrm{d}t \right) \\
&= -2\varepsilon\rho_0 \left( 1 + \frac{2}{R^2}g\big(u_0(z_r)\big) \right) +\frac{T_0K}{4R} + \rho^\varepsilon(T_0,z_r) \\
&\stackrel{\eqref{eq_epsT0}}{>} \rho^\varepsilon(T_0,z_r).
\end{align*}
By \eqref{Eq_Intersection}, the evolution of $\Gamma_0^\varepsilon$ hence developed a self-intersection. \qedhere
\end{enumerate}
\end{proof}

\subsection{Properties of the Concentration}

In this section we discuss that the so called ``concentration'' $c: \Gamma \rightarrow \R$ really satisfies the most important properties of a physical concentration. First, the concentration should describe the distribution of a quantity whose mass is conserved. Second, the concentration should always be non-negative. We will analyze these features in the setting of Assumptions \ref{ass_concentration}. 

\begin{theorem}[Conservation of Mass]\label{Massenerhaltung} $\phantom{x}$ \\
Suppose Assumptions \ref{ass_concentration} are valid. There exists a constant $m \in \R$ (specifying the mass of the quantity whose concentration is described by the function $c$) such that 
\begin{align*}
\int_{M} c(t,p) \, \mathrm{d}\mathcal{H}^d(p) = m
\end{align*}
holds for all $t \in [0,T]$. 
\end{theorem}

\begin{proof}
With the help of the transport theorem (Proposition \ref{transporttheorem}) and Gauß' theorem on closed hypersurfaces (Proposition \ref{partialintegration}), we have
\begin{align*}
\frac{\mathrm{d}}{\mathrm{d}t} \int_{M} c \, \mathrm{d}\mathcal{H}^d 
= \int_{M} \partial^\square c - cHV \, \mathrm{d}\mathcal{H}^d
= \int_{M} \Delta_{\Gamma} \big(G'(c)\big) \, \mathrm{d}\mathcal{H}^d 
= 0. & \qedhere
\end{align*}
\end{proof}

A concentration always is non-negative. Therefore, we show in the following theorem that non-negativity of the concentration is conserved. Even more, if an initially non-negative concentration is not the zero-function, then it instantly turns strictly positive. 

\begin{theorem}[Positivity of the Concentration]\label{concentration_pos} $\phantom{x}$ \\
Suppose that Assumptions \ref{ass_concentration} are valid with $G''>0$. 
\begin{enumerate}
\item[(i)]
Let $c(0) \geq 0$ on $M$. Then $c(t) \geq 0$ holds on $M$ for all $t \in [0,T]$.
\item[(ii)]
Let $c(0) \geq 0$ on $M$ with $c(0) \not\equiv 0$. Then $c(t)>0$ holds on $M$ for all $t \in (0,T]$.
\end{enumerate}
\end{theorem}

\begin{proof}
On account of \eqref{eq_Intro_GLS1}, we have
\begin{align*}
\Delta_\Gamma G'(c) + cVH
&= A : D_\Gamma^2 c + B \cdot \nabla_\Gamma c + Cc
\end{align*}
with 
\begin{align*}
A \defr G''(c) \Id, \phantom{xx} 
B \defr G'''(c) \nabla_\Gamma c \phantom{xx} \text{and} \phantom{xx}
C \defr g(c)H^2.
\end{align*}
The assumptions and $H \in C^0\big([0,T] \times M\big)$ by Definition \ref{mc_def} imply continuity of the functions $A: [0,T] \times M \rightarrow \R^{(d+1) \times (d+1)}$, $B: [0,T] \times M \rightarrow \R^{d+1}$ and $C: [0,T] \times M \rightarrow \R$. Moreover, the matrix $A$ clearly is symmetric and, due to $G''>0$, also positive definite on $[0,T] \times M$. 
With this, we define $w \defr -c$ as well as
\begin{align*}
\mathcal{L}w \defr -\partial^\square w + A : D_\Gamma^2 w + B \cdot \nabla_\Gamma w + Cw.
\end{align*}
In particular, Assumptions \ref{ass_maxprinzip} are satisfied. We have 
\begin{align*}
\mathcal{L}w 
= -\mathcal{L}c
= \partial^\square c - \big( \Delta_\Gamma G'(c) + cVH \big)
= 0 \text{ on } [0,T] \times M
\end{align*}
and $w(0) = -c(0) \leq 0$ on $M$. With Proposition  \ref{MaxPrinzip_Folg}, $c(t)=-w(t) \geq 0$ follows on $M$ for all $t \in [0,T]$ and there exists $t_0 \in [0,T]$ with 
\begin{align*}
c(t) = -w(t) = 0 \text{ on } M \text{ for all } t \in (0,t_0] 
\phantom{x} \text{ and } \phantom{x}
c(t) = -w(t) > 0 \text{ on } M \text{ for all } t \in (t_0,T].
\end{align*}
Under the assumptions of (ii), we have $c(0) \not\equiv 0$. By Theorem \ref{Massenerhaltung}, thus $c(t) \equiv 0$ can not hold for any $t>0$. Consequently, we then have $t_0=0$ and $c(t)>0$ follows on $M$ for all $t \in (0,T]$.
\end{proof}

The concentration does not only stay non-negative, but the minimal concentration even increases monotonically. Additionally, if the hypersurface is mean convex, the increase of the minimal concentration even is strictly monotonic. 

\begin{theorem}[Growth of the Minimal Concentration] $\phantom{x}$ \\
Suppose Assumptions \ref{ass_concentration} are valid with $G''>0$ and $g>0$. Let $c_{\min}: [0,T] \rightarrow \R, \, c_{\min}(t) \defr \min_{M} c(t,\cdot)$ be the minimum function of $c$ on $M$.
\begin{enumerate}
\item[(i)]
Let $c(0) \geq 0$. Then, $c_{\min}$ is monotonically increasing.
\item[(ii)]
Let $\Gamma$ be of regularity $(C^2$-$C^2) \cap (C^1$-$C^4)$. Furthermore, let $H(0) \geq 0$, $c(0) \geq 0$ and $c(0) \not\equiv 0$. Then, $c_{\min}$ is strictly increasing. 
\end{enumerate}
\end{theorem}

\begin{proof}
With Hamilton's trick (see \cite[Lemma 2.1.3]{Mantegazza}), $c_{\min}: (0,T) \rightarrow \R$ is well-defined and Lip\-schitz continuous. In particular, it is differentiable almost everywhere, and in every time $t \in (0,T)$ in which $c_{\min}$ is differentiable, we have
\begin{align*}
\partial_t c_{\min \, |t} = \partial_t c_{|(t,p)}
\end{align*}
where $p \in M$ is an arbitrary point with $c(t,p) = c_{\min}(t)$. For such a point $p$, we have $\nabla_\Gamma c(t,p) = 0$ and $\Delta_\Gamma c(t,p) \geq 0$ by Lemma \ref{surface_Hessian_Extremepoints} and then
\begin{align*}
\partial^\square c_{|(t,p)}
= \partial^\circ c_{|(t,p)} - V^{\text{tot}}_{\Gamma \, |(t,p)} \cdot \nabla_\Gamma c_{|(t,p)}
= \partial^\circ c_{|(t,p)}
= \partial_t c_{|(t,p)}
\end{align*}
follows with Definition \ref{timederivatives_Def}. For every point $p \in M$ with $c(t,p) = c_{\min}(t)$ we thus have
\begin{align*}
\partial_t c_{\min \, |t} 
&= \partial^\square c_{|(t,p)} 
= \Delta_\Gamma G'(c) + cVH_{\, |(t,p)} \\
&= G''(c) \Delta_\Gamma c + G'''(c) \big| \nabla_\Gamma c \big|^2 + g(c)H^2c_{\, |(t,p)} 
\geq g(c)H^2c_{\, |(t,p)}. 
\end{align*}
The assumptions in (i) and Theorem \ref{concentration_pos}(i) imply $c(t) \geq 0$ for all $t \in [0,T]$ and therefore $\partial_t c_{\min} \geq 0$ follows almost everywhere. Hence, $c_{\min}$ is monotonically increasing. \\
The assumptions in (ii) as well as Section \ref{Chap_MeanConvex} and Theorem \ref{concentration_pos}(ii) yield $H(t)>0$ and $c(t)>0$ for every $t \in (0,T]$ such that $\partial_t c_{\min} > 0$ follows almost everywhere. Hence, $c_{\min}$ is strictly increasing. 
\end{proof}

\appendix
\section{Appendix}

\subsection{Basic Results for Hypersurfaces}

We gather some well-known results for hypersurfaces.

\begin{lemma}[Surface Derivatives in Extreme Points]\label{surface_Hessian_Extremepoints} $\phantom{x}$ \\
Let $\Sigma=\theta(M)$ be a $C^2$-immersed hypersurface and let $f \in C^2(M,\R)$ have a maximum in $p \in M$. Then, we have 
\begin{align*}
\nabla_\Sigma f(p) = 0
\phantom{xx} \text{ and } \phantom{xx}
D_\Sigma^2 f(p)\leq 0.
\end{align*}
\end{lemma}

A proof of these intuitive statements can be found in \cite[Lemma 2.36]{Buerger}.

\begin{proposition}[Closed Surfaces have non-vanishing Mean Curvature]\label{mc_zero} $\phantom{x}$ \\
Let $\Sigma = \theta(M) \subset \R^{d+1}$ be a $C^2$-immersed closed hypersurface. The mean curvature on $\Sigma$ is not the zero function.
\end{proposition}

A closed, embedded hypersurface cannot have vanishing mean curvature. With an ana\-lo\-gous argumentation, the same statement holds in the immersed case (see \cite[Proposition 2.44]{Buerger}).

\begin{proposition}[Normal Time Derivative of the Mean Curvature]\label{mc_formulas} $\phantom{x}$ \\
We assume $\Gamma$ to be a \mbox{$(C^2$-$\phantom{.}C^2) \cap (C^1$-$\phantom{.}C^4)$}-evolving immersed hypersurface with unit normal $\nu$, mean curvature $H$ and normal velocity $V$. Then, 
\begin{align*}
\partial^\square H = \Delta_\Gamma V + V \big| \nabla_\Gamma \nu \big|^2 
\end{align*}
holds on $\Gamma(t)$ for every $t \in [0,T]$. 
\end{proposition}

A proof of this statement can be found in \cite[Lemma 39(ii)]{Gpde}.

\begin{proposition}[Gauß' Theorem on Closed Hypersurfaces]\label{partialintegration} $\phantom{x}$ \\
Let $\Sigma=\theta(M) \subset \R^{d+1}$ be a $C^2$-immersed closed hypersurface and let $F \in C^1(M,\R^{d+1})$ with $F(p) \in T_p \Sigma$ for every $p \in M$ as well as $f \in C^2(M,\R)$ and $g \in C^1(M,\R)$. Then we have 
\begin{align*}
\int_\Sigma \Div_\Sigma F \, \mathrm{d}\mathcal{H}^d = 0
\phantom{xx} \text{ and } \phantom{xx}
\int_\Sigma g \Delta_\Sigma f \, \mathrm{d}\mathcal{H}^d = - \int_\Sigma \nabla_\Sigma g \cdot \nabla_\Sigma f \, \mathrm{d}\mathcal{H}^d.
\end{align*}
\end{proposition}

This theorem is well-known for embedded hypersurfaces and can be proven with the same arguments as in \cite[Theorem 5.1.7, which relies on Lemmas 5.1.5 and 5.1.6 therein]{Baer} for the immersed case.

\begin{proposition}[Transport Theorem]\label{transporttheorem} $\phantom{x}$ \\
Let $\Gamma$ be a $C^1$-$\phantom{.}C^2$-evolving immersed closed hypersurface with reference surface $M \subset \R^{d+1}$, mean curvature $H$ and normal velocity $V$. For a function $f \in C^1([0,T] \times M)$, we have
\begin{align*}
\frac{\mathrm{d}}{\mathrm{d}t} \int_{\Gamma(t)} f \, \mathrm{d}\mathcal{H}^d
= \int_{\Gamma(t)} \partial^\square f - f H V \, \mathrm{d}\mathcal{H}^d.
\end{align*}
\end{proposition}

This result is well-known for embedded surfaces but as integration is not defined locally, it does not transfer directly to the case of immersed surfaces. The more subtle argumentation for the transfer relies on the first statement in \cite[Theorem 32]{Gpde} and on Gauß' theorem (cf. \cite[Proposition 2.58]{Buerger}).

\section*{Acknowledgements}

The second author is grateful for the funding from the DFG Research Training Group 2339 Interfaces, Complex Structures, and Singular Limits.

\bibliographystyle{amsplain}
\bibliography{literature}

\end{document}